\documentclass[12pt,leqno]{article}
\usepackage{amsmath,amssymb,amsfonts,amsthm}
\usepackage{graphicx}
\usepackage{color}

\newcommand{\Walk}{{\mathfrak W}}

\newcommand{\adistance}{a_{\star}}
\newcommand{\B}{{\mathcal B}}
\newcommand{\G}{{\mathcal G}}
\newcommand{\Set}{{\mathcal S}}

\newcommand{\whp}{with high probability}
\newcommand{\V}{\mathcal{V}}
\newcommand{\W}{\mathcal{W}}
\newcommand{\N}{{\mathcal N}}
\newcommand{\A}{{\mathcal A}}

\newcommand{\Bin}[2]{\textrm{Bin}\left(#1,#2\right)}
\newcommand{\E}{\mathbb{E}}
\newcommand{\Var}{{\rm Var}}
\newcommand{\Pra}[1]{\Pr\left\{#1\right\}}
\newcommand{\PraG}[1]{\text{Pr}_{\G}\left\{#1\right\}}

\newcommand{\Gnm}[1]{\mathcal{G}_{#1}\left(n,m,p\right)}
\newcommand{\Gnmp}{\Gnm{}}
\newcommand{\Bnm}[1]{\mathcal{B}\left(n,m,#1\right)}
\newcommand{\Bnmp}{\Bnm{p}}

\newcommand{\LARGEv}{\text{LARGE}}
\newcommand{\SMALLv}{\text{SMALL}}
\newcommand{\eps}{\varepsilon}
\newcommand{\dO}{d_0}
\newcommand{\djeden}{d_1}

\newcommand{\Dki}{D(k,i)}
\newcommand{\EDki}{\bar{D}(k,i)}
\newcommand{\EDkiO}{\bar{D}(k,i_0)}

\newcommand{\Dk}{D(k)}
\newcommand{\Dstar}{D^{\star}(k,i_0)}

\newcommand{\EDk}{\bar{D}(k)}

\newcommand{\EG}{\E_{\G}}
\newcommand{\dist}{{\rm{dist}}}

\newcommand{\EW}{{\mathfrak{W}}}
\newcommand{\ov}{{\overline{v}}}
\newcommand{\ok}{{\overline{k}}}
\newcommand{\op}{{\overline{p}}}

\newcommand{\onul}{{\overline{0}}}

\theoremstyle{plain}
\newtheorem{thm}{Theorem}

\newtheorem{lem}[thm]{Lemma}

\newtheorem{fact}[thm]{Fact}
\theoremstyle{definition}

\title{The cover time of a random walk in affiliation networks}
\author{Mindaugas Bloznelis\footnote{Institute of Computer Science, Vilnius University, 03225 Vilnius, Lithuania.
    }
	\ \
	Jerzy Jaworski\footnote{Faculty of Mathematics and Computer Science, Adam Mickiewicz University, 60-614 Pozna\'n, Poland} 
	\ \
	Katarzyna Rybarczyk\footnotemark[2]
	\footnote{supported by NCN (National Science Center) grant 2014/13/D/ST1/01175}	
}

\date{}

\begin{document}

\maketitle
%\centerline{2019 10 21} 

\begin{abstract}

Many known networks have structure of affiliation networks, where each of $n$ network's nodes (actors) selects 
an attribute set from a given collection of $m$ attributes and two nodes (actors) establish adjacency relation
whenever they share a common attribute. We study behaviour of the random walk on such networks. For that purpose we use commonly used model of such networks -- random intersection graph. We establish  the cover time of  the simple random walk on the binomial random intersection graph 
${\cal G}(n,m,p)$ at the connectivity threshold and above it. 
We consider the range of $n,m,p$ where
the typical attribute is shared by (stochastically) bounded number of actors.
\end{abstract}

{\bf keywords: } 
complex networks, 
information networks, 
%random intersection graph, 
random graph,
random walk, 
cover time.

\smallskip
MSC-class: 05C81,
05C80, 
05C82,
91D30. 

\section{Introduction and results}\label{Introduction}

\subsection{Motivation}

Many known networks, such as for example sensor networks with random key predistribution, internet network, WWW, social networks, have apparent or hidden structure of affiliation networks (see \cite{GuillaumeLatapy2004}). In an affiliation network  
 each node (actor) is prescribed a finite set of 
 attributes and two actors establish adjacency relation 
 whenever they share  a common attribute. 
 For example, in the  sensor networks with random key predistribution two sensors are connected, when they have a randomly prescribed key in common,
 in film actor network two actors are adjacent if they 
 have played in the same movie, in 
 the collaboration network two scientists are adjacent if 
 they have co-authored a publication. The structure of the affiliation network with actors and attributes might be hidden (\cite{GuillaumeLatapy2004}). However it is plausible that the structure of actors and attributes determines the form of various networks including complex networks such as social networks, internet network, WWW. 

Random walks are a standard tool of data collection in large 
networks. 
 We analyse the behaviour of a random walk on affiliation networks.  
 More precisely we study the cover time of such random walk on the binomial random intersection graph, which is commonly used theoretical model of affiliation networks. 
 The cover time, i.e., the expected time needed to visit all  
 vertexes of the network, is a fundamental characteristic of a random 
 walk. We establish the first order asymptotics for the cover 
 time of the binomial  random intersection graph. Our results are 
 mathematically rigorous. They are inspired by the 
 fundamental study by Cooper and Frieze 
 \cite{CooperFrieze2007,CooperFrieze2008}
 of the cover 
 time of the Erd\H os-R\'enyi  random graph, where edges are 
 inserted independently at random. 
 
One of the motivations of our work is an analysis of the influence of clustering on the behaviour of a random walk. This is because an important 
feature of affiliation  networks and complex networks is the clustering 
property, i.e.,  
the tendency of nodes to cluster together by forming 
relatively small groups with a high density of ties 
within a group
\cite{NewmanStrogatzWatts2002}.
We are interested in whether and how the {\it clustering 
	property affects the random walk performance}, i.e. for how long a random walk might be ``stuck'' in a small tight ``community'' of nodes.

It is interesting and technically challenging problem to determine the cover time of the connected component of more general models of random intersection graphs. The most intersecting models would be for example wireless sensor networks \cite{DiPetro2004,YoganMakowski2012} or random intersection graph models with non-vanishing clustering coefficient and power law degree distribution \cite{BloznelisGJKR2015,Spirakisetal2013}.

\subsection{The model}

In this article we analyse the random intersection graph model introduced in \cite{karonski1999}, see also \cite{FriezeKaronski2016}. It was studied in the context of random walks already in \cite{SpirakisCover2013} however only partial results were obtained. 

In this section we define the random graph model $\G(n,m,p)$. We let $n,m\to\infty$ and use the asymptotic notation 
$o(\cdot)$, $O(\cdot)$, $\Omega(\cdot)$, $\Theta(\cdot)$, $\asymp$ explained in \cite{JansonLuczakRucinski2001}.
%We use the phrase ``{\whp}'' to say with %probability tending to one as $n$ tends to %infinity.
The phrase ``{\whp}'' will mean that the
probability of the  event under consideration tends to one as $n,m$ tend to infinity.

The binomial random intersection graph $\G(n,m,p)$ with
 vertex set $\V$ of size $n$ is defined using  an auxiliary set of attributes $\W$ of size $m$. 
 Every vertex $v\in \V$ is assigned a random set of attributes 
 $\W(v)\subset \W$ and two vertexes $u,v\in \V$ become adjacent whenever they share a common attribute, i.e.,
 $\W(u)\cap \W (v)\not=\emptyset$. We assume that  events $w\in \W(v)$ are independent and  have (the same) probability $p$. 
% We say that $w\in \W$ and  $v\in \V$ are linked if $w\in \W(v)$.

Noting that every set 
$\V(w)=\{v:\, w\in \W(v)\}\subset V$ of vertexes  
sharing an attribute $w\in \W$ induces a clique in the 
intersection graph  we can represent $\G(n,m,p)$ as a 
union of $m$  randomly located cliques. The sizes  
$|\V(w)|$, $w\in \W$, of these cliques are independent 
binomial random variables with the common  
distribution $\Bin{n}{p}$. 
%This representation shows
%that the neighboring adjacency relations are %statistically dependent and explains the clustering %property \cite{Deijfen},  \cite{Bloznelis2013}.

In this paper we assume 
that 
the expected size of the typical clique $\E |\V(w)|=np=\Theta(1)$.
Furthermore, we focus on 
 the connectivity threshold, which for $\G(n,m,p)$ is defined 
by the relation
\begin{equation}\label{2019-09-30}
 mp(1-e^{-np})=\ln n,
\end{equation}
see \cite{Katarzyna2017}, \cite{Singer1995PhD}. Namely, for 
$a_n=mp(1-e^{-np})-\ln n$ tending to $+\infty$ the probability that 
$\G(n,m,p)$ is connected is $1-o(1)$. For 
 $a_n\to-\infty$ this probability is $o(1)$.
 Hence we will assume below that $m=\Theta(n\ln n)$. 

\subsection{Main result and related results}

We recall that given a connected graph $G$ with the vertex set $V$, $|V|<\infty$,
the cover time $C(G)=\max_{u\in V}C_u$, where $C_u$ is the expected number of steps needed by the simple random walk starting from vertex $u$ to visit all the vertexes of $G$. We prove the following theorem.

\begin{thm}\label{ThmMain}
Let $n, m\to+\infty$. Let $p=p(m,n)>0$ and 
$c=c(m,n)>1$ be such that $np=\Theta(1)$, $c=O(1)$
 and 
\begin{equation}\label{c}
mp(1-(1-p)^{n-1})=c\ln n
\ \ \
\text{and} 
\ \ \
(c-1)\ln n\to\infty.
\end{equation}
Then {\whp} 
the 
cover time of a random walk on $\Gnmp$ is 
\begin{equation}\label{2019-02-21}
	(1+o(1))\cdot\ln \left(\frac{np}{\ln\left(\frac{c-1}{c}(e^{np}-1)+1\right)}\right)\cdot \frac{np}{1-e^{-np}}\cdot c n \ln n.
\end{equation}
\end{thm}

In an earlier paper \cite{SpirakisCover2013},
 the $O(\ln n)$ upper bound on the mixing time of $\Gnmp$ has been shown for the model's parameters $p=4m^{-1}\ln n$ and 
$m=n^{\alpha}$, where $\alpha\le 1$ is fixed. Note that for $m=n^{\alpha}$, $\alpha\le 1$, the connectivity threshold is at $p=m^{-1}\ln n$. Interestingly, 
\cite{SpirakisCover2013}
lowerbounds the conductance, but of the related bipartite graph (see $\Bnmp$ in Section 2 below) instead of 
$\Gnmp$. In Theorem~\ref{ThmMain} we concentrate on the case where $m\asymp n\ln n$ therefore the range of parameters $n,m,p$ considered here does not intersect with that of 
\cite{SpirakisCover2013}. 

The result presented in Theorem~\ref{ThmMain} concerns theoretical model of affiliation networks. In what follows we would like to compare this result with known results concerning networks with independent links (Erd\H{o}s--R\'enyi random graph). Note that under assumptions of the Theorem~\ref{ThmMain}
$\Gnmp$ is just above the connectivity threshold. Indeed, for $np=\Theta(1)$ the quantities 
$(1-p)^{n-1}$ and  $e^{-np}$ are $O(n^{-1})$ close.
It is convenient to represent the connectivity threshold
in terms of the edge density
(i.e.,  the probability that $u,v\in \V$ are adjacent
in $\G(n,m,p)$). We denote by $p_I$ the edge density of $\G(n,m,p)$.  A simple calculation shows that
$p_I=1-(1-p^2)^m=mp^2(1+O(p^2))$.
Furthermore, at the connectivity threshold (\ref{2019-09-30}) the edge density $p_I$ approximately equals 
$(1-e^{-np})^{-1}np(n^{-1}\ln n)=:{\hat p}_I$. Note that
%We remark that 
the latter expression 
differs by the factor 
$(1-e^{-np})^{-1}np >1$ from the edge density $n^{-1}\ln n$ that defines the connectivity threshold for the
Erd\H os-R\'enyi random graph on $n$ vertexes, see  \cite{JansonLuczakRucinski2001}. 

For the edge density $p_I$ defined by formula 
(\ref{c}) we have
$p_{I}=(c+O(n^{-1})){\hat p}_I$. Furthermore  
$(c-1)\ln n\to+\infty$ implies that $p_I>{\hat p}_I$.
Moreover, we obtain from (\ref{2019-02-21}) that the random intersection graph with edge density $p_I=c{\hat p}_I$, $c>1$, has the cover time
\begin{equation}\label{comparison}
(1+o(1))\lambda c{\hat p}_In^2,
\qquad
{\text{where}}
\qquad
\lambda=\ln \left(\frac{np}{\ln\left(\frac{c-1}{c}(e^{np}-1)+1\right)}\right).
\end{equation}
It is  interesting to compare (\ref{2019-02-21}) with the respective  cover time  of
the  Erd\H os - R\'enyi random graph  
$\G(n,q)$. While comparing two different random graph models (both just above their respective connectivity thresholds) one may set the reference point to be a connectivity threshold. Then  (\ref{comparison}) is compared  with the 
cover time of $\G(n,q)$ for the edge density $q$ being 
just above the
connectivity threshold ${\hat q}_B=n^{-1}\ln n$. 
The cover time of $\G(n,c{\hat q}_B)$, $c>1$, as $n\to+\infty$  is given by the formula, see \cite{CooperFrieze2007},
\begin{equation}\label{comparison2}
(1+o(1))\ln(c/(c-1))c{\hat q}_Bn^2.
\end{equation}
%Let us compare (\ref{comparison}) and (\ref{comparison2}). 
Interestingly,
 $\lambda=\lambda(np,c)$ is always less than $\ln(c/(c-1))$ (we show this  in Sect 4.2 below). Furthermore, we have $\lambda(np,c)\to \ln(c/(c-1))$ as $np\to 0$. 
The later fact is not surprising as $\G(n,m,p)$ 
is the union of iid random cliques 
$\V (w)$, $w\in \W$. For $np\to0$ the absolute majority of the cliques are either empty or induce a single edge of $\G(n,m,p)$. Therefore in this range $\G(n,m,p)$ starts looking  similar to a union of iid edges. But a union of iid edges represents the Er\H os-R\'enyi
 random graph.

On the other hand, in the case where edge densities of 
$\G(n,m,p)$ and $\G(n,q)$ are the same, the cover time 
of $\G(n,m,p)$ is always larger than $\G(n,q)$. To see 
this
we set the edge density of $\G(n,q)$ to be 
$q=c{\hat p}_I={\bar c}{\hat q}_B$, $c>1$, where
we write for short ${\bar c}=np(1-e^{-np})^{-1}c$.
By (\ref{comparison2}), the cover time 
of $\G(n,{\bar c}{\hat q}_B)$ is
\begin{displaymath}
(1+o(1))
\ln ({\bar c}/({\bar c}-1)){\bar c} {\hat q}_B
=
\ln ({\bar c}/({\bar c}-1))c{\hat p}_I.
\end{displaymath}
In Section 4.2 below we show that $\lambda>\ln({\bar c}/({\bar c}-1))$. 
Therefore, given the edge density 
(namely, $c{\hat p}_I={\bar c}{\hat q}_B$, $c>1$), the cover time of the random intersection graph 
%$\G(n,m,p)$ 
is always larger than that of the binomial graph.
% $\G(n,q)$.
This can be explained by the fact that the abundance 
of cliques in $\G(n,m,p)$
may slow down the random walk considerably.
It is interesting to trace the relation between the ratio $\lambda/\ln({\bar c}/({\bar c}-1))$ and the expected size  of the typical clique $np=\E|{\cal V}(w)|$. In Figure \ref{FigComparison2} we present a numerical plot of the function $np\to \lambda(np,c)/\ln ({\bar c}/({\bar c}-1))$ for $0\le np\le 30$ and  $c=1.1$, $c=2$ and $c=10$.

 \begin{figure}
 	\includegraphics[width=14cm]{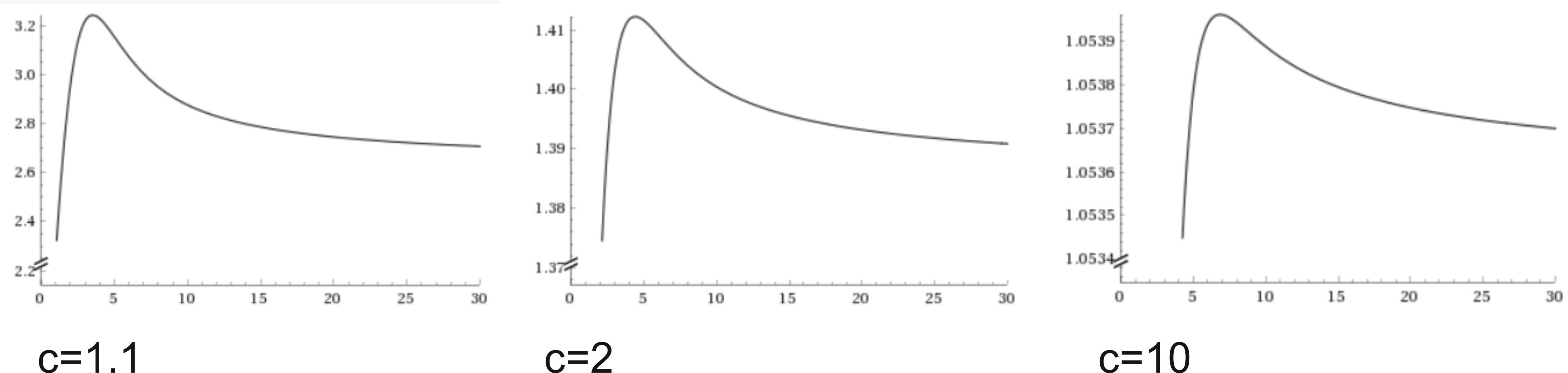}
 	\caption{
	The plot of 
	$np\to \lambda(np,c)/\ln ({\bar c}/({\bar c}-1))$.
% 		$
% 		\lambda_{x,c}/\bar{\lambda}_{x,c}
% 		$
% 		for $x\in [0,30]$
}
\label{FigComparison2}
 \end{figure}

From a technical point of view the main reason that makes the cover times so different 
is  that at the connectivity thresholds 
the 
degree distributions of $\Gnmp$ and  
$\G(n,q)$
(and consequently the stationary distributions of respective random walks) differ a lot.
Therefore we believe that results on the degree distribution
of $\Gnmp$ obtained in Sect.\ref{SectionDegrees} below
 are interesting in its own right as they give an insight into the structure of random intersection graphs.

\subsection{Outline of the proof}

 We conclude this section with the outline of the proof.
 In the proof we use heuristics developed in works of Cooper and Frieze  \cite{CooperFrieze2007,CooperFrieze2008}. However in the case of random intersection graphs we need to develop some new results in order to understand the behavior of a random walk on graphs with more clustering.  

First we discuss that {\whp} $\Gnmp$ has 
certain properties which are listed in 
Section~\ref{SectionProperties}. We  
call an intersection graph with such 
properties \emph{typical}. 
%Some of the proofs of the properties are standard. 
%These are postponed to Appendix.  
%Furthermore, the conductance property that guarantees $O(\ln n)$ upper bound for the mixing time is shown in a separate paper  \cite{conductance}.

In Section~\ref{SectionDegrees} we establish  the properties of the degree distribution  that, in fact,  define the cover time asymptotics  (\ref{2019-02-21}). We study the degree distribution in an interesting regime, where $np=\Theta(1)$. This makes the analysis more interesting and technically involved.

In  the following sections we consider the simple random walk 
$\EW_u$ 
on a typical $\G$ starting at a vertex $u\in \V$.  Given $\G$ and $u$ we denote by 
$C_u$ the expected time taken for the walk to visit every vertex of $\G$. Here we use ideas from \cite{CooperFrieze2008}. In particular, in order to get an upper bound on $C_u$, we use equation
 (42) of 
\cite{CooperFrieze2008} 
\begin{equation}\label{EqCu}
C_u\le t+1+\sum_v\sum_{s\ge t}
\PraG{\A_s(v)}, \text{for each $u\in\V$ and all $t\ge T$},
\end{equation} 
where $T$ is an integer and $\PraG{\A_s(v)}$ is probability that $\Walk_v$ on a given graph $\G$  does not visit $v$ in steps $T, T+1, \ldots, s$. For the lower bound we need to establish the second moment of the number of unvisited vertexes by time $t$. For that we need to evaluate $\PraG{\A_s(v)\cap \A_s(v')}$.

In particular in Section~\ref{SectionReturns-}  we concentrate on approximating the values $\PraG{\A_s(v)}$ and $\PraG{\A_s(v)\cap \A_s(v')}$. 
For that we study  the expected number of returns of a random walk. We note that in contrast to the Erd\H os-R\'enyi graph the typical vertex of $\G(n,m,p)$  may belong to quite a few small cycles.
 This makes our analysis of the return probabilities far more involved and interesting.

Finally, in Section \ref{SectionCoverTime}, combining the results of  Sections~\ref{SectionDegrees} and~\ref{SectionReturns-},  we
establish matching upper and lower bounds for the cover time that hold uniformly over the typical graphs.

\section{Typical graphs}\label{SectionProperties}

\subsection{Intersection graphs - notation}

In order to define needed properties we first introduce some notation.
For any graph $G$ we denote by ${\cal E}(G)$ its edge 
set. We remark that the intersection graph $\Gnmp$ with 
the vertex set $\V$ and attribute set $\W$ is 
obtained from
the
 bipartite graph $\Bnmp$ with bipartition $(\V,\W)$, 
 where each vertex  $v\in\V$ and  each attribute
 $w\in\W$ are linked independently and with probability $p$. Any two  vertexes
 $u,v\in \V$ are adjacent in the intersection graph 
 whenever they share a common neighbor in the 
 bipartite graph.
By $\G$ and $\B$ we denote realized instances of random graphs $\Gnmp$ and  
$\Bnmp$ and always assume that ${\cal G}$ {\it is defined by} ${\cal B}$. 
 For convenience, edges in $\B$  we call {\sl links}. Cycles 
 %$v_1w_1v_2\ldots w_{k-1} v_kw_kv_1$ ($v_i\in\V,w_i\in\W$) 
 in $\B$ we  call $\B$--cycles. 
 It is easy to see that for $k\ge 3$ any $\B$-cycle  $v_1w_1v_2\ldots w_{k-1} v_kw_kv_1$ 
 ($v_i\in\V,w_i\in\W$) defines the cycle  $v_1v_2,\ldots v_kv_1$ in $\G$. Furthermore, any induced cycle $v_1v_2\ldots v_kv_1$ $(k\ge 3)$ 
 in $\G$ is defined by some $\B$--cycle $v_1w_1v_2\ldots w_{k-1}v_kw_kv_1$.   
%$\B$--cycles of length four $v_1w_1v_2w_2v_1$ are %present if  $v_1$ and $v_2$ have at least two common %attributes.
Note that in $\G$ there might be many other cycles 
besides those defined by $\B$--cycles. We denote by  
$\dist(v,v')$ the distance between $v$ and $v'$ in $\G$.
 We additionally set notation:\\
the set of attributes of $v\in \V$, which contribute to at least one edge in $\G$
$$
\W'(v)=\{w\in\W(v): |\V_{\G}(w)|\ge 2\};
$$
the set of vertexes which have attribute $w$
$$
\V(w)=\{v\in\V: w\in \W(v) \}\};
$$
the set of neighbors of a 
vertex $v\in\V$ and the set of neighbors of $\Set\subseteq \V$
$$
\N(v)=\{v'\in\V\setminus\{v\}: \W(v)\cap\W(v')\neq\emptyset\},
\quad \N(\Set)=\bigcup_{v\in \Set}\N(v)\setminus \Set.
$$
Furthermore, for $i=0,1,2,\dots$ we denote by $\N_i(v)$ the 
 set of vertexes at distance $i$ from $v$ in $\G$.
 So that $\N_0(v)=\{v\}$ 
 and $|\N(v)|=|N_1(v)|=:\deg(v)$ is the degree of $v\in\V$.
 We denote by $\Dk$ 
 the number of vertexes of degree $k$ in $\G$ and 
 $\Dki$ the number of vertexes $v$ of degree $k$ such that $|\W'(v)|=i$. 
Let
$$\SMALLv=\{v:|\W'(v)|\le 0.1\ln n\}
\quad\text{ and }
\quad
\LARGEv=\V\setminus\SMALLv.$$
Vertices belonging to  $\SMALLv$ ($\LARGEv$) we call 
{\it small} ({\it large}).
%
%For convenience 
%we introduce the short-hand notation
Let
\begin{equation*}
\dO=mp(1-(1-p)^{n-1}),
\quad\quad
\djeden=nmp^2,
\qquad
\quad 
\Delta=\lceil\max\{4d_0,12d_1\}\rceil.
\end{equation*}
Note that  $\dO=\E |\W'(v)|$, $\djeden$ is the 
approximate expected degree 
and
$np=\Theta(1)$ implies
\begin{equation}\label{2019-02-25+1}
\dO
\asymp 
\djeden 
\asymp 
\Delta 
\asymp
\ln n.
\end{equation}
In our calculations 
we will frequently use the fact that 
$$mp(1-e^{-np})=d_0(1+O(n^{-1}))=(1+O(n^{-1}))c\ln n.
$$ 
By ${\Pr}_{\cal G}$ and $\EG $ we denote the conditional probability and expectation given ${\cal G}$. By $c',c''$ we denote positive constants that can be different in different places.
Throughout the proof the inequalities hold for $n$ large enough. If it does not influence the result, we consequently omit $\lfloor\cdot\rfloor$ and $\lceil\cdot\rceil$ for the sake of clarity of presentation.

\subsection{Typical intersection graphs}

\begin{lem}\label{LemProperties} Let $m,n\to+\infty$.
Assume that conditions of Theorem 1 hold. Then there  exists a constant $\adistance>0$ that may depend on the sequences $\{p(m,n)\}$ and 
$\{c(m,n)\}$, but not on $n,m$  such that {\whp} 
$\Gnmp$
and $\Bnmp$ have properties {\bf P0--P8} listed below.
\end{lem}

\begin{proof} The proofs of the most of the properties listed below are either known (for example {\bf P0}) or standard. 
Their proofs are to be found in Appendix.
%We omit them in order to concentrate on the most important part. 
Property {\bf P8} is shown in Section \ref{SectionDegrees} and 
{\bf P2}~is proved in the accompanying paper \cite{conductance}.
\end{proof}

\begin{itemize}
	\item[{\bf P0}] $\G$ is connected and has at least one odd cycle.
	\item[{\bf P1}]  $\Big||{\cal E}(\G)|-\frac{n^2mp^2}{2}\Big| \le n^{1/2}\ln n.$
	\item [{\bf P2}] 	
	$$
	\min_{\Set,|\Set|\le n/2}\frac{e(\Set,\bar{\Set})}{2e(\Set,\Set) + e(\Set,\bar{\Set})}>\frac{1}{50},
	$$
	where,  for any $\Set\subseteq\V$,  $e(\Set,\Set)$ is the number of edges induced by set $\Set$ and $e(\Set,\bar{\Set})$ is the number of edges between $\Set$ and $\V\setminus\Set$.
	\item [{\bf P3}] 
	For all $v\in\V$  we have
	$|\W'(v)|\le \Delta$ and $\deg(v)\le \Delta$.
	\item [{\bf P4}] Each adjacent pair  $v,v'\in \V$ shares at most $\max \{2np; 4\}\frac{\ln n}{\ln\ln n}$  common neighbors.
%, $\acommon:=\max \{2np; 4\}$.
	
\item [{\bf P5}] Every $v\in \V$  has at 
least $|\W'(u)|-1$
neighbors in $\G$. Every  $v\in \LARGEv$ has at least 
$(\ln n)/11$ neighbors in $\G$.

\item [{\bf P6}] For every $v$ and 
$1\le i\le \adistance \ln n/\ln\ln n$ each  vertex 
from $\N_i(v)$ has at most two neighbors in 
$\N_{i-1}(v)$.
	
	\item [{\bf P7}] Any two small vertexes are at least 
$\adistance \ln n/\ln\ln n$ links apart. 
Each small vertex and each 
$\B$--cycle of length at most 
$\adistance \ln n/\ln\ln n$ are at least 
$\adistance \ln n/\ln\ln n$ links apart.
Any two $\B$--cycles of length at most 
$\adistance \ln n/\ln\ln n$ are at least $\adistance \ln n/\ln\ln n$ links apart.

%	\todo{Do we use the second part of it %somewhere? For sure we use the fact that %small $\B$--cycles are far apart - which is %not include here}
    \item [{\bf P8}] Introduce the numbers
\begin{equation}\label{EqDefDkiDk}
\EDki
=
\frac{{k\brace i}}{k!}
(mpe^{-np})^i
(np)^k
n^{1-c},
\qquad
\EDk
=
\sum_{i=1}^{k}
\EDki.
\end{equation}
Here $c$ is from (\ref{c}) and 
${k\brace i}$ denotes  
    Stirling's number of the second kind. We remark that $\EDk$,~$\EDki$ are approximations to $\E \Dk$, 
    $\E \Dki $, see (\ref{EqEDkiDk}) below. 
Define  
	\begin{align*}
	K_1&=\{1\le k\le 20:  \EDk \le \ln\ln n\};\\
	K_2&=\{21\le k\le \Delta :  \EDk \le (\ln n)^2\};\\
	K_3&=\{1,2,\ldots,\Delta\}\setminus ( K_1\cup K_2).
	\end{align*}
	\\
	{\bf P8a} 
For $k\in K_1$ we have $D(k)\le (\ln\ln n)^2$, 
for $k\in K_2$ we have $D(k)\le (\ln n)^4$, and	
for $k\in K_3$ we have 
$\frac{1}{2}\EDk\le \Dk\le \frac{3}{2}\EDk$.
	
	{\bf P8b} If $(c-1) \ge \ln^{-1/3} n$ then $\Dk=0$ for all  $k\le \ln^{1/2} n$.
	\\
	{\bf P8c} 
Define
	\begin{align*}
&
i_0
=
\lceil
(c-1)\ln n
\rceil,
\qquad
k_0
=
\lceil 
i_0 \cdot \max\{10\,e^{np}(e^{np}-1),2\}
\rceil,
\\
	&I=\{k:i_0\le k\le k_0\text{ and } \EDkiO \ge  i_0^{2}\}.
	\end{align*}
	Let $\Dstar$ be the number of vertexes $v\in \V$ such that $|\W'(v)|=i_0$, $\deg(v)=k$, and $v$ is at distance at least $\ln n/(\ln\ln n)^3$ from any other vertex $v'\in \V$ with $|\W'(v')|=i_0$.
	We have $I\not=\emptyset$ and 
	$\Dstar\ge \EDkiO/2$ for  $k\in I$.
\end{itemize}

We call an instance $\G$ of $\Gnmp$ typical if it has properties {\bf P0-P8}.

\section{Vertex degrees in $\Gnmp$}\label{SectionDegrees}

In this section we study vertex degrees in $\Gnmp$
and prove that {\whp} $\Gnmp$ has property {\bf P8}. 
We consider vertexes $v\in \V$ with degrees at most
$\Delta$
%In the proof we only consider degrees 
%from the interval $[1,\Delta]$ 
and  with
$|\W'(v)|\le \Delta$. Note that $\Delta\le c'\ln n$ for some constant $c'>0$, see (\ref{2019-02-25+1}).

In the proof of {\bf P8} we will use the first moment and the second moment method. Therefore we will need to establish the expected value and the variance of the number of vertexes with degree $k$.
First we will prove that
uniformly in $i\le k\le \Delta$ 
\begin{equation}\label{EqEDkiDk}
\E \Dki 
=
\EDki
\bigl(1+O(n^{-1}\ln^4 n)\bigr)
\quad
\ 
\text{and}
\quad
\
\E \Dk 
=
\EDk
\bigl(1+O(n^{-1}\ln^4 n)\bigr). 
\end{equation}
We will not study directly the variance of $\Dk$. Instead we will consider an auxiliary random variable 
\begin{equation*}
D''(k) =\sum_{i=i_k}^{k}D(k,i),
\text{ where }i_k:=\min\{1;\lceil k/\ln\ln n\rceil\}\text{ for } k=1,2,\dots, \Delta.  
\end{equation*}
(Note  that in a connected $\G$
for each vertex $v$ of degree $k\in [1,\ln\ln n]$
we have  $|\W'(v)|\ge 1=i_k$.) And then we will prove that  
 \begin{eqnarray}
 \label{EqDbisVariance}
 \Var (D''(k))
 &
 =
 &
 \E (D''(k)(D''(k)-1))
 +
 \E D''(k)
 -
 (\E D''(k))^2
 \\
 \nonumber
 &
 =
 &
 (\E D''(k))^2O(n^{-0.9})
 +O(n^{-1}\ln^3n)
 +
 \E D''(k).
 \end{eqnarray}
We might consider $D''(k)$ instead of $\Dk$ as
\begin{equation}
\label{2019-04-07+++}
\Pr\bigl\{ \forall v\in \V \
{\text{we  have}}
\ 1\le \deg(v)\le \Delta
\
{\text{and}}
\ 
i_{\deg(v)}\le |\W'(v)|\le \deg(v)\bigr\}
=1-o(1)
\end{equation}
which we will prove as well.

After establishing the first and the second moment of $\Dk$ and $D''(k)$, resp., we will proceed with the proof of the fact that $\Gnmp$ {\whp} has property {\bf P8}.
 
%In what follows we will use some additional notation. 
%Notation introduced in this section do not extend to other sections.
In the proofs we  
use in several places the following relations for $1\le t<i\le k$ and $1\le h\le k-i+1$
\begin{equation}\label{2019-04-01}
\binom{k}{i}
\frac{i^{k-i}}{2}
\ge 
{k\brace i}
\ge 
{k\brace t}
k^{2(t-i)}, 
\qquad
{k\brace i}
\ge 
i^{h-1}
{k-h+1\brace i}
\ge 
i^{h-1}
{k-h\brace i-1}.
\end{equation}
The first inequality is shown in \cite{RennieDobson}.
The second one  is equivalent to 
${k\brace j}/{k\brace j-1}\ge k^{-2}$, $j\ge 1$, which
follows from the fact that $j\to {k\brace j}/{k\brace j-1}$ decreases, see 
\cite{CanfieldPomerance}, combined with  
${k\brace k}/{k\brace k-1}={\binom{k}{2}}^{-1}$. The third and fourth inequalities follow by multiple application of
the recursion relation
${n+1\brace r}=r{n\brace r}+{n\brace r-1}$.  
 
\subsection{Configurations in $\Bnmp$ and their probabilities}

Note that $v\in\V$ has degree $k$ in $\G$ if in $\B$ by which it is defined there are sets $\N(v)\subset \V$ and $\W'(v)\subset \W$ such that: all attributes from $\W'(v)$ are linked to $v$, each attribute from $\W'(v)$ is linked to some vertex in $\N(v)$ and is not linked to any vertex from $\V\setminus (\N(v)\cup \{v\})$, each vertex from $\N(v)$ is linked to some attribute in $\W'(v)$. In order to prove \eqref{EqEDkiDk} we need to count probabilities of such configurations in $\Bnmp$.

\noindent

For the purpose of establishing these probabilities we introduce some notation.  
Notation introduced in Section~\ref{SectionDegrees} do not extend to other sections.
By $A_i, A_i'$ and $B_k, B_k'$ we denote subsets of
$\W$ and  $\V$ of sizes $i$ and $k$ respectively. 
In what follows it is convenient to think of $B_k$ and $A_i$ as realised neighborhoods ${\cal N}(v)=B_k$ and
$\W'(v)=A_i$ of some $v\in \V$ 
($B_k'$ and $A_j$ refer to respective neighborhoods of another vertex $u\in \V$).
We say that  $A_i$ covers
$B_k$ if each node from $B_k$ is linked to some vertex 
from $A_i$ in ${\cal B}(n,m,p)$. 
For $i\le k$ we call $A_i$ a cover of $B_k$ if $A_i$ covers $B_k$ and no proper subset of $A_i$ covers $B_k$  (note that $A_i$ may cover $B_k$ not being a cover of $B_k$). A cover $A_i$ is an economic cover (e-cover) if there are exactly $k$ links between $A_i$ and $B_k$. The probability 
that $A_i$ is an e-cover of $B_k$ 
is
\begin{equation}
\label{2019-03-28++1}
(1-p)^{ik-k}
{\bar p}_{k,i},
\qquad
{\text{where}}
\qquad
{\bar p}_{k,i}={k\brace i} i!p^k.
\end{equation}
For $A_t\subset A_i$ consider a configuration of links 
between $A_i$ and $B_k$ such that $A_t$ is an e-cover 
of 
 $B_k$ and each node belonging to $A_i\setminus A_t$ is 
linked to a single vertex from $B_k$. We call such a 
configuration basic ($A_i/B_k$ basic configuration). 
Let ${\cal A}_{A_i,B_k}=\{
{\cal B}(n,m,p)$ contains  an $A_i/B_k$ basic configuration  as a subgraph$\}$ and denote by 
$p^*_{k,i}=\Pr\{{\cal A}_{A_i,B_k}\}$ its probability. For $i\le k$ we have
\begin{equation}
\label{2019-04-15+1}
%\sum_{j=1}^i{\binom{i}{j}}{\bar p}_{k,j}(kp)^{i-j}
%=
(1-p)^{ik-k}
{\bar p}_{k,i}
\le
p^*_{k,i}
\le
{\bar p}_{k,i}+\delta_{k,i},
\qquad
{\text{where}}
\qquad
\delta_{k,i}
=
\sum_{t=1}^{i-1}{\binom{i}{t}}{k\brace t}t!p^k
(kp)^{i-t}.
\end{equation}
The first inequality is obvious and the second one follows  by the union bound:
 $\binom{i}{t}$ counts e-covers  
$A_t\subset A_i$ of 
size $t$, ${k \brace t}t!$ counts  configurations of 
$k$ links between $A_t$ and $B_k$ that realize e-cover 
$A_t$.
Furthermore, $k^{i-t}$ upper bounds the number of ways to link members of $A_i\setminus A_t$ to arbitrary vertexes of $B_k$.
Note that 
$\delta_{k,i}$
is negligible compared to ${\bar p}_{k,i}$. 
By (\ref{2019-04-01}),
\begin{equation}\label{2019-03-30}
\delta_{k,i}
\le 
{\bar p}_{k,i}
\sum_{t=1}^{i-1}{\binom{i}{t}}(k^3p)^{i-t}
\le 
{\bar p}_{k,i}\bigl((1+k^3p)^i-1\bigr)
\le 
c'' {\bar p}_{k,i}ik^3p.
\end{equation}
Hence, we have  uniformly in $1\le i\le k\le \Delta$ that
\begin{equation}\label{2019-03-28++5}
p^*_{k,i}
=
{\bar p}_{k,i}(1+O(n^{-1}\ln^{4}n)).
\end{equation}
%Assuming that  $B_k$ and $A_i$ are realised %neighborhoods ${\cal N}(v)=B_k$ and 
%$\W'(v)=A_i$ of some $v\in \V\setminus B_k$
% we calculate the probability of the event
For $v\in V\setminus B_k$ define the event
${\cal A}_{v, A_i,B_k}= \bigl\{$each node from $A_i$ is linked to $v$,  there are no 
links 
between $A_i$ and $\V\setminus (B_k\cup \{v\})$, and
 none 
element of $\W\setminus A_i$ belongs to $\W'(v)
\bigr\}$. Its probability 
\begin{eqnarray}\label{2019-03-28++9}
p'_{k,i}
:=
\Pr\{{\cal A}_{v,A_i,B_k}\}
&
=
&
p^i(1-p)^{i(n-k-1)}
(1-p+p(1-p)^{n-1})^{m-i}
\\
\nonumber
&
=
&
p^ie^{-inp}e^{-d_0}(1+O(n^{-1}\ln^2n)).
\end{eqnarray}

Similarly, for  $B_k, B_k'$, $A_i, A_j$ 
and $u\not=v$ such that  
$(B_k\cup B_k')\cap \{u,v\}=\emptyset$
and 
$A_i\cap A_j=\emptyset$
the probability 
that events
${\cal A}_{v, A_i,B_k}$ and ${\cal A}_{u, A_j,B_k'}$ occur simultaneously 
\begin{displaymath}
p'_{k,i,j}(0)
:=
\Pr\{{\cal A}_{v, A_i,B_k}\cap{\cal A}_{u, A_j,B_k'}\}
=
p^{i+j}(1-p)^{(i+j)(n-k-1)}
\bigl((1-p)^2+2p(1-p)^{n-1}\bigr)^{m-i-j}.
\end{displaymath}
Note that $B_k$ and $B_k'$ 
 may intersect.
%not need to be distinct.
Here $((1-p)^2+2p(1-p)^{n-1})^{m-i-j}$ is the probability that none element from 
$\W\setminus (A_i\cup A_j)$ belong to 
$\W'(v)\cup\W'(u)$.
Furthermore, for
$|A_i\cap A_j|=r\in\{1,2\}$ and $v\in B'_k$, $u\in B_k$
the probability 
$\Pr\{
{\cal A}_{v, A_i,B_k}\cap{\cal A}_{u, A_j,B_k'}
\}$ is at most
\begin{displaymath}
p'_{k,i,j}(r)
:=
p^{i+j}(1-p)^{(i+j-2r)(n-k-1)}
\bigl((1-p)^2+2p(1-p)^{n-1})^{m-i-j+r}.
\end{displaymath}
We remark that $p'_{k,i,j}(r)\le c'p'_{k,i,j}(0)$, for $r=1,2$ and
\begin{equation}\label{2019-04-16++1}
p'_{k,i,j}(0)=
p^{i+j}e^{-(i+j)np}e^{-2d_0}(1+O(n^{-1}\ln^2n))
=
p'_{k,i}p'_{k,j}(1+O(n^{-1}\ln^2n)).
\end{equation}

\subsection{Proof of \eqref{EqEDkiDk}}

\smallskip

We start with proving the first part of (\ref{EqEDkiDk}).
Let $v\in \V$ and $1\le i\le k$. Given $A_i$, 
$B_k\subset\V\setminus \{v\}$,
 we have
$\Pr\{\W'(v)=A_i,\, \N(v)=B_k \}=p^*_{k,i}p'_{k,i}$. By the union rule,
\begin{eqnarray}\label{2019-03-28++9x1}
p_{k,i}
&
:=
&
\Pr\{\W'(v)|=i,\, |\N(v)|=k\}
=
{\binom{n-1}{k}}{\binom{m}{i}}
p^*_{k,i}p'_{k,i}
\\
\nonumber
&
=
&
{k\brace i}
\frac{(np)^k}{k!}
\bigl(mpe^{-np}\bigr)^i
e^{-d_0}
\bigl(1+O(n^{-1}\ln^{4}n)\bigr).
\end{eqnarray}
In the last step we invoked  (\ref{2019-03-28++1}), 
(\ref{2019-03-28++5}), (\ref{2019-03-28++9})
%\end{document}
and used the approximations
\begin{eqnarray}
\label{2019-04-02}
&&
\binom{n-1}{k}
=
\frac{n^k}{k!}
\bigl(1+O(k^2/n)\bigr),
\qquad
\binom{m}{i}
=
\frac{m^{i}}{i!}
\bigl(1+O(i^2/m)\bigr),
%\\
%\label{2019-03-28++9}
%&&
%(1-p)^{i(n-k-1)}
%=
%\exp(-inp+ikp+kp+O(inp^2))
%=
%e^{-inp}\bigl(1+O(n^{-1}\ln^2 n)\bigr),
%\\
%\label{2019-03-28++10}
%&&
%\bigl(1-p+p(1-p)^{n-1}\bigr)^{m-i}
%=
%\exp\bigl(-mp(1-(1-p)^{n-1})+O(mp^2+ip)\bigr)
%\\
%\nonumber
%&&
%\qquad
%\qquad
%\qquad
%\qquad
%\qquad
%\quad
%\,
%=
%e^{-\dO}\bigl(1+O(mp^2+ip)\bigr)=n^{-c}%\bigl(1+O(n^{-1}\ln n)\bigr).
\end{eqnarray}
Finally, 
using (\ref{2019-03-28++9x1})
 we evaluate the expectation
\begin{equation}\nonumber
\E \Dki
= 
np_{k,i}
=
\EDki
\bigl(1+O(n^{-1}\ln^4 n)\bigr).
\end{equation}

\noindent
Now we show the second part of (\ref{EqEDkiDk}).
We  split
\begin{equation}\label{2019-04-06}
\E D(k)=\sum_{1\le i\le k}\E D(k,i)+R, 
\qquad
R=\sum_{r\ge 1}\E D(k,k+r)
\end{equation}
 and we prove that $R=\E D(k,k)O(n^{-1}\ln^2 n)$.  
Given $r\ge 1$, $v\in \V$ and
$B_k\subset \V\setminus \{v\}$,  any 
instance $\B$ 
favouring the event $|\W'(v)|=k+r$, $\N(v)=B_k$ contains
an $A_k/B_k$ basic 
configuration for some
$A_k$. 
In addition, there is $A_r'\subset\W\setminus A_k$  such 
that $\W'(v)=A_k\cup A_r'$.
Hence, the
probability 
$p^*_k=\Pr\{|\W'(v)|>k, \, \N(v)=B_k\}$ is at most 
\begin{equation}\label{2019-03-21}
\binom{m}{k}\sum_{r\ge 1}\binom{m-k}{r}
p^*_{k,k}p'_{k,k+r}(pk)^r.
\end{equation}
Here $\binom{m}{k}\binom{m-k}{r}$ counts pairs 
$A_k,\, A_r'$ and $(pk)^r$ upper bounds the probability that each node from $A_r'$ is linked to some vertex from $B_k$.
Using (\ref{2019-03-21}) we bound
\begin{equation}\label{2019-03-21+1}
R=
n
{\binom{n-1}{k}}
p^*_k
\le
n
{\binom{n-1}{k}}
{\binom{m}{k}}
p^*_{k,k}
p'_{k,k}
R',
\end{equation}
where
\begin{equation}\label{2019-03-21+2}
R'
\le 
\sum_{r\ge 1}\binom{m-k}{r}\frac{p'_{k,k+r}}{p'_{k,k}}(pk)^r
\le
\sum_{r\ge 1}(mp^2k)^r
=
O(n^{-1}\ln^2 n).
\end{equation}
Here we used $p'_{k,k+r}\le p^rp'_{k,k}$.
From the first part of  
(\ref{EqEDkiDk}),(\ref{2019-03-28++9x1}),
(\ref{2019-03-21+1}),
(\ref{2019-03-21+2})
we obtain $R=O(n^{-1}\ln^{2}n)\E D(k,k)$. This combined with \eqref{2019-04-06} and the first part of \eqref{EqEDkiDk} imply the second part of \eqref{EqEDkiDk}.

\subsection{Proof of \eqref{2019-04-07+++}}

Let
 \begin{equation}\nonumber
 \quad
 \
 X=\sum_{\ln\ln n< k\le \Delta}
 \
 \sum_{i=1}^{i_{k}-1}D(k,i),
 \quad
 \
 Y=\sum_{k=1}^{\Delta}\sum_{r\ge 1}D(k,k+r).
 \end{equation}
We will prove that for some $\varepsilon>0$ 
(depending on the sequence $np=\Theta(1)$) 
we have
\begin{equation}\label{2019-04-07}
\E Y
=O(n^{-\varepsilon})
\qquad
{\text{and}}
\qquad
\E X=O(\ln^{-10}n).
\end{equation}

\noindent
Now we prove the first part of (\ref{2019-04-07}).
(\ref{2019-04-06}) 
combined with the first relation of 
(\ref{EqEDkiDk}) imply
$\E Y
=
O(n^{-1}\ln^{2}n)
\sum_{1\le k\le\Delta}{\bar D}(k,k)$.
Here
\begin{displaymath}
\sum_{1\le k\le\Delta}{\bar D}(k,k)
=
n^{1-c}\sum_{k\ge 1}(nmp^2e^{-np})^k/k!
\le
n^{1-c}e^{nmp^2e^{-np}}.
\end{displaymath}
Invoking $mp(1-e^{-np})(1+O(n^{-1}))=c\ln n$, see 
(\ref{c}), we write the right side in the form
$n^{1-c}e^{np(e^{np}-1)^{-1}c\ln n}(1+O(n^{-1}\ln n))$.
For $np=\Theta (1)$ this quantity is 
$O(n^{1-\varepsilon})$ for some 
$\varepsilon>0$, since the ratio $np/(e^{np}-1)<1$ is bounded away from $1$.
Hence $\E Y=O(n^{-\varepsilon}\ln^2n$). 

\noindent
Now we prove the second bound of (\ref{2019-04-07}). 
 By 
 %the first relation of 
 (\ref{EqEDkiDk}),
% we have
\begin{displaymath}
\E D(k,i)
\le
2\EDki
\le
\frac{(ke)^ii^{k-2i}}{k!}
(mpe^{-np})^i(np)^{k}n^{1-c}=:f(k,i).
\end{displaymath}
Here we used $2{k\brace i}\le \binom{k}{i}i^{k-i}$, 
see
(\ref{2019-04-01}), 
and $\binom{k}{i}\le (ke/i)^i$. 
The inequality $f(k,i+1)/f(k,i)\ge 1$ implies that 
$i\to  f(k,i)$ 
increases for $1\le i<i_k$.
Hence,
\begin{equation}
\label{2019-03-23-2}
\sum_{1\le i\le i_k-1}\E \Dki
<
i_kf(k,i_k).
\end{equation}
%Next we estimate $i_kf(k,i_k)\frac{(np)^k}{k!}%=:e^{A_k}$.
Note that our assumptions 
$np=\Theta(1)$ and (\ref{c}) 
 imply $np\le c_1$ and $mpe^{-np}\le c_2\ln n$ for some $c_1,c_2>0$. Using these inequalities we estimate, 
for $k\le \Delta\le c'\ln n$,
 %Invoking Stirling's formula  we obtain
\begin{eqnarray}\nonumber
\ln\bigl(i_kf(k,i_k)\bigr)
&
\le
&
i_k+i_k\ln k+(k-2i_k+1)\ln i_k+i_k\ln(c_2\ln n)+k\ln c_1-\ln k!
\\
\nonumber
%\label{2019-03-23-1}
&
=
&
-\ln k!+k\ln i_k+O(k)=-k\ln\ln\ln n+O(k).
%\le 
%-0.5k\ln\ln\ln n.
\end{eqnarray}
Combining the latter bound with  (\ref{2019-03-23-2})
 we obtain
\begin{equation}\nonumber
%\label{2019-03-28++7}
\E X
\le 
\sum_{k> \ln\ln n}
i_kf(k,i_k)
\le
\sum_{k> \ln\ln n}e^{-0.5k\ln\ln\ln n}
=
o(\ln^{-10}n).
\end{equation} 
Finally, we observe that  {\bf P0}, {\bf P3} and 
(\ref{2019-04-07}) imply \eqref{2019-04-07+++}.

\subsection{Proof of \eqref{EqDbisVariance}}

Here we  upper bound the variance $\Var D''(k)$.
Given $\{u,v\}\subset \V$, let $p''_{-}$ be the probability that $|\W'(u)\cap W'(v)|\ge 3$ and  
$p_{+}''$ be the probability that $\deg(u)=\deg(v)=k$,
$|\W'(u)\cap W'(v)|\le 2$ and
$i_k\le |\W'(u)|\le k, i_k\le |\W'(v)|\le k$. 
We have  
\begin{equation}\label{2019-04-04+1}
n(n-1)p''_{+}
\le
\E 
\Bigl( D''(k)(D''(k)-1)
\Bigr)
\le n
(n-1)(p''_{+}+p''_{-}),
\end{equation}
where
$p''_{-}\le \binom{m}{3}p^6\le c'n^{-3}\ln^3n$ is negligible.
Let us evaluate $p''_{+}$. We split 
\begin{equation}\label{2019-04-04}
p''_{+}
=
\sum_{0\le r\le 2}\sum_{i_k\le i,j\le k}p_{k,i,j}(r),
\end{equation}
where 
$p_{k,i,j}(r)$ stands for the probability of the event
\begin{equation}\label{2019-04-15}
%p_{k,i,j}(r)
%=
%\Pr\bigl
\{
|\W'(u)|=j, 
|\W'(v)|=i,
|\W'(u)\cap \W'(v)|=r,
\deg(u)=\deg(v)=k\bigr\}.
\end{equation}
We show below that 
uniformly in $1\le i_k\le i,j\le k\le \Delta$
\begin{equation}\label{2019-04-03+7}
p_{k,i,j}(0)=p_{k,i}p_{k,j}
\Bigl(
1+O\Bigl(\frac{\ln n}{n}\Bigr)\Bigr),
\qquad
p_{k,i,j}(r)
\le O(n^{-r+0.1})p_{k,i}p_{k,j},
\quad
r=1,2.
\end{equation}
From (\ref{EqEDkiDk}), (\ref{2019-04-04+1}), (\ref{2019-04-04}) and (\ref{2019-04-03+7}) we obtain \eqref{EqDbisVariance}
\begin{equation}\nonumber
\E (D''(k)(D''(k)-1))=(\E D''(k))^2(1+O(n^{-0.9}))
+
O(n^{-1}\ln^3 n).
\end{equation}
We are left with showing (\ref{2019-04-03+7}). 
Let $r=0$. For $u\not=v$,  $B_k, B_k'$, $A_i, A_j$ such 
that  
$(B_k\cup B_k')\cap \{u,v\}=\emptyset$
and 
$A_i\cap A_j=\emptyset$ we have
\begin{displaymath}
\Pr
\bigl\{
{\cal A}_{A_i,B_k}
\cap 
{\cal A}_{A_j,B_k'}
\cap 
{\cal A}_{v, A_i,B_k}
\cap 
{\cal A}_{u,A_j,B_k'}
\bigr\}
=p^*_{k,i}p^*_{k,j}p'_{k,i,j}(0).
\end{displaymath}
Summing over $A_i,A_j$ with
$A_i\cap A_j=\emptyset$  and over (not necessarily distinct) $B_k,B_k'$
we obtain 
\begin{displaymath}
p_{k,i,j}(0)
= 
{\binom{n-1}{k}}^2
{\binom{m}{i}}
{\binom{m-i}{j}}
p^*_{k,i}p^*_{k,j}p'_{k,i,j}(0)
=
p_{k,i}p_{k,j}\bigl(1+O(n^{-1}\ln^2 n)\bigr).
\end{displaymath}
In the last step we used
(\ref{2019-04-16++1})
and (\ref{2019-03-28++9x1}).

%For $r=1,2$ we show that uniformly in
%$i_k\le i,j\le k\le \Delta$ 
%\begin{equation}
%\label{2019-04-02+2}
%p_{k,i,j}(r)
%\le 
%p_{k,i}p_{k,j}O(n^{-0.9}).
%\end{equation}
Let $r=1$. For $1\le i,j\le k$ we split 
$p_{k,i,j}(1)=\sum_{0\le h\le k-1}p_{k,i,j}(1,h)$, 
where  $p_{k,i,j}(1,h)$ is the probability of the event (\ref{2019-04-15})
intersected with the event that $w=A_i\cap A_j$ has $h$ neighbors in $\V\setminus\{u,v\}$.
We have
\begin{eqnarray}
\label{2019-04-01+3}
p_{k,i,j}(1,h)
&
\,
\le
&
\,
m{\binom{m-1}{i-1}}{\binom{m-i}{j-1}}
\cdot
{\binom{n-2}{h}}
{\binom{n-h-2}{k-h-1}}^2
\\
\nonumber
&
&
\cdot
\,
p^h
p^*_{k-h-1,i-1}p^*_{k-h-1,j-1}
\,
p'_{k,i,j}(1).
\end{eqnarray}
The first line  counts
triplets
$\{w\}, A_{i}, A_{j}$ such that 
$\{w\}=A_i\cap A_j$ and  triplets 
$B_h,B_k,B'_k$ such that $u\in B_k\subset \V\setminus\{v\}$, $v\in B'_k\subset \V\setminus\{u\}$ and $B_h\subset B_k\cap B'_k$.
Furthermore, 
$p^h$ is the probability that $w$ is 
linked to each vertex of $B_h$,
$p^*_{k-h-1,i-1}$ is the probability that $A_i\setminus\{w\}$ covers 
$B_k\setminus(B_h\cup\{u\})$,
$p^*_{k-h-1,j-1}$ is the probability that $A_j\setminus\{w\}$ covers 
$B'_k\setminus(B_h\cup\{v\})$. For $h=k-1$ we put 
$p^*_{0,s}:=(p(k-1))^{s}$ so that
$p^*_{0,i-1}$ ($p^*_{0,j-1}$)
upper bounds the probability that 
$B_k\setminus\{u\}$ covers $A_{i}\setminus \{w\}$ 
($B'_k\setminus\{v\}$ covers $A_{j}\setminus \{w\}$).

\noindent
Now we show that
\begin{equation}\label{2019-04-01+1}
p^*_{k-h-1,i-1}
\le 
c'{k\brace i}\frac{(i-1)!}{i^h}p^{k-h-1},
\quad
0\le h<k-1,
\quad
{\text{and}}
\quad
p^*_{0,i}
\le 
{k\brace i}i^{i-k}(pk)^{i-1}.
\end{equation}
The second inequality follows from ${k\brace i}i^{i-k}\ge {i\brace i}=1$.
To prove the first one we use the bound, cf. (\ref{2019-04-15+1}),
\begin{equation}
\label{2019-04-17}
p^*_{l,j}
\le
\sum_{t=1}^{j\wedge l}
{\binom {j}{t}}
{l\brace t}t!p^{l}(pl)^{j-t},
\qquad
j,l\ge 1.
\end{equation} 
\noindent
Let $h<k-1$. Let moreover $\tau=(i-1)\wedge(k-h-1)$. 
For $\tau=i-1$ we obtain from (\ref{2019-04-15+1}),
(\ref{2019-03-30}) that
\begin{displaymath}
p^*_{k-h-1,i-1}
=
{\bar p}_{k-h-1,i-1}(1+o(1))
\le 
c'{k-h-1\brace i-1}(i-1)!p^{k-h-1}.
\end{displaymath}
Now the inequalities 
${k\brace i}
\ge 
i^{h}{k-h\brace i}
\ge 
i^h{k-h-1\brace i-1}$, 
see (\ref{2019-04-01}), imply
(\ref{2019-04-01+1}). For $\tau=k-h-1$ we 
apply (\ref{2019-04-17}) and invoke 
${\tau\brace t}\le t^{2(\tau-t)}{\tau\brace\tau}=t^{2(\tau-t)}$, see  (\ref{2019-04-01}). 
We obtain
\begin{displaymath}
p^*_{k-h-1,i-1}
\le
p^{\tau}
\sum_{t=1}^{\tau}
{\binom{i-1}{t}}
t!\tau^{2(\tau-t)}(p\tau)^{i-t-1}.
\end{displaymath}
Then using $p\tau^3=o(1)$ 
we upper bound the right side by
\begin{displaymath}
p^{\tau}
(p\tau)^{i-\tau-1}
(i-1)!
\sum_{t=1}^\tau(p\tau^3)^{\tau-t}
\le 
c'p^{\tau}
(p\tau)^{i-\tau-1}
(i-1)!.
\end{displaymath}
Furthermore, we multiply the right side by
${k\brace i}i^{i-k}\ge1$, see (\ref{2019-04-01}), and 
use $p\tau i\le 1$ to get
(\ref{2019-04-01+1}).
Proof of (\ref{2019-04-01+1}) is complete.

\noindent
In the next step we invoke (\ref{2019-04-01+1}) in 
(\ref{2019-04-01+3})
and apply  (\ref{2019-04-02}). We obtain 

\begin{eqnarray}
\label{2019-04-02+1}
&&
p_{k,i,j}(1,h)
\le
c'm^{i+j-1}
\frac
{(np)^{2k}}
{(k!)^2}
{k\brace i}{k\brace j}
p'_{k,i,j}(1)
S^*_{k,i,j}(h)
\le c'' p_{k,i}p_{k,j}
S^*_{k,i,j}(h),
\end{eqnarray}
where
\begin{displaymath}
S^*_{k,i,j}(h)
:=
\frac{((k)_{h+1})^2}
{h!(np)^{h+2}(ij)^{h}},
\quad
h<k-1,
\quad
{\text{and}}
\quad
S^*_{k,i,j}(k-1)
:=
\frac{k!k}{(np)^{k+1}}
\frac{(pk)^{i+j-2}i^{i-k}j^{j-k}}{(i-1)!(j-1)!}.
\end{displaymath}
In the last step of (\ref{2019-04-02+1}) we used
$p'_{k,i,j}(1)
\le 
c' p'_{k,i,j}(0)
\le 
c''p'_{k,i}p'_{k,j}$, see (\ref{2019-04-16++1}), and 
(\ref{2019-03-28++9x1}). 
A calculation shows that
$\sum_{0\le h\le k-1}S^*_{k,i,j}(h)=O(n^{0.1})$ uniformly in $i_k\le i,j\le k\le c'\ln n$.
Furthermore, we have 
$p'_{k,i,j}(1)\le c' p'_{k,i,j}(0)
\le c''p'_{k,i}p'_{k,j}$, see (\ref{2019-04-16++1}). 
Hence,
(\ref{2019-04-02+1}) imply 
(\ref{2019-04-03+7}) for $r=1$. For $r=2$ the 
proof of (\ref{2019-04-03+7}) is much the 
same.
 
\subsection{Proof of {\bf P8}}

 Now we prove {\bf P8a} and {\bf P8b}. We sketch only the proof of {\bf P8c}. In view of 
 (\ref{2019-04-07+++}) it suffices to show that 
 {\bf P8} holds for $D''(k)$.
 
 \noindent
 {\it Proof of}\, {\bf P8a}. 
 The second part of {\bf P8a} follows from 
 %
 %the second relation of 
 (\ref{EqEDkiDk}) by Markov's inequality.
 To prove the first part 
 we show that
 %By the union bound,
 \begin{displaymath}
 1
 -
 \Pr
 \Bigl\{
 \frac{1}{2}{\bar D}(k)\le D''(k)\le \frac{3}{2}{\bar D}(k),\,
 k\in K_3
 \Bigr\}
 \le
 \Pr\{\cup_{k\in K_3}{\cal B}_k\}
 \le
 \sum_{k\in K_3}
 \Pr\{{\cal B}_k\}=o(1).
 \end{displaymath}
 Here we write for short
 ${\cal B}_k=\{|D''(k)-{\bar D}(k)|>{\bar D}(k)/2\}$.
 The first two inequalities are obvious. To prove the 
 last bound we show that
 $\Pr\{{\cal B}_k\}\le O(n^{-0.9})+2/{\bar D}(k)$.
 From
 (\ref{EqEDkiDk}),
 (\ref{2019-04-07})
 we obtain
 \begin{displaymath}
 |{\bar D}(k)-\E D''(k)|\le |{\bar D}(k)-\E D(k)|+\E X+\E Y =O(n^{-1}\ln^2n){\bar D}(k)+O(\ln^{-9}n).
 \end{displaymath} 
 For $k\in K_3$ 
 we have
 $0.9\le \E D''(k)/{\bar D}(k)\le 1.1$. 
 Now, by Chebyshev's inequality and 
 (\ref{EqDbisVariance}), 
 \begin{displaymath}
 \Pr\{{\cal B}_k\}
 \le 
 \Pr\{|D''_k-\E D''(k)|\ge {\bar D}(k)/3\}
 \le 
 O(n^{-0.9})+1.1/{\bar D}(k).
 \end{displaymath}
 
 \noindent
 {\it Proof of}\,  {\bf P8b}. We need to show  that 
 $\Pr\{\deg(v)>\ln^{1/2}n,\, \forall v\in \V\}=1-o(1)$.
 In view of {\bf P5} it suffices to prove 
 that
 $p_0
 :=
 \Pr\{\exists v\in \V: \,|\W'(v)|< 2\ln^{1/2}n\}=o(1)$.
 Note that each $|\W'(v)|$ has binomial distribution with mean
 $d_0=c\ln n$,  see (\ref{c}). By the union bound and 
 Chernoff's inequality, see (2.6) in \cite{JansonLuczakRucinski2001}, we have
 \begin{eqnarray}
 \nonumber
 p_0
 \le 
 n\Pr
 \Bigl\{
 |\W'(v)|\le \frac{2d_0}{\sqrt{\ln n}}
 \Bigr\}
 \le
 n
 \exp
 \Bigl\{
 -\Bigl(
 1
 -
 \frac{2}{\sqrt{\ln n}}
 +\frac{2}{\sqrt{\ln n}}
 \ln\frac{2}{\sqrt{\ln n}}
 \Bigr)
 d_0
 \Bigr\}
 =
 o(1).
 \end{eqnarray}
 
 \noindent
 {\it Proof of} {\bf P8c}.
 Let us prove that $I\not=\emptyset$.  We begin with showing auxiliary inequality (\ref{2019-09-29 formules numeracija nereikalinga}), see below.
 Given $y>0$, $q>1$ and integer $i>1$, let 
 $r=\lceil i+iqy\rceil$. We have 
 \begin{displaymath}
 \sum_{k\ge r}
 {k\brace i}
 \frac{y^k}{k!}
 \le
 \frac{y^i}{i!}
 \sum_{k\ge r}
 \frac{(yi)^{k-i}}{(k-i)!}
 \le
 \frac{y^i}{i!}
 \frac{(yi)^{r-i}}{(r-i)!}\frac{q}{q-1}
 \le
 \frac{y^i}{i!}
 \frac{(e/q)^{r-i}}{\sqrt{2\pi} \sqrt{r-i}}
 \frac{q}{q-1}.
 \end{displaymath}
 In the first step we use 
 ${k\brace i}\le {\binom{k}{i}}i^{k-i}$. In the second step we upper bound the series by the geometric series $1+q^{-1}+q^{-2}+\cdots$
 using the fact that the ratio of two consecutive terms is at most $q^{-1}$.
 The last inequality follows by Stirling's approximation. Choosing $q=2e$ we upper bound the right side by $y^i/(2i!)$. Combining this bound with the  identity 
 $\sum_{k\ge i}{k\brace i}y^k/k!=(e^y-1)^i/i!$ 
 and inequality $e^y-1>y$
 we obtain
 for any $r\ge \lceil i+i2ey\rceil$
 \begin{equation}\label{2019-09-29 formules numeracija nereikalinga}
 \sum_{k=i}^{r}{k\brace i}\frac{y^k}{k!}
 \ge 
 \frac{(e^y-1)^i}{i!}
 -\frac{1}{2}\frac{y^i}{i!}
 \ge 
 \frac{1}{2}\frac{(e^y-1)^i}{i!}.
 \end{equation}
 For 
 $i=i_0$, $r=k_0$ and $y=np$
 this inequality implies
 \begin{equation}\label{2019-04-11}
 \sum_{k=i_0}^{k_0}
 {\bar D}(k,i_0)
 \ge
 \frac{1}{2}
 \frac{(mp(1-e^{-np}))^{i_0}}{i_0!}
 n^{1-c}
 \ge
 \frac{1+O(i_0/n)}{2e\sqrt{i_0}}
 \Bigl(\frac{c\ln n}{i_0}\Bigr)^{i_0}.
 \end{equation} 
 In the second step we used (\ref{c}) and Stirling's approximation. Furthermore, by the assumption 
 $(c-1)\ln n\to +\infty$, we have for large $n$ that
 $c/(c-1)\ge (c\ln n)/i_0>2c/(2c-1)$. We conclude that the right side of (\ref{2019-04-11}) grows exponentially in $i_0$. This proves $I\not=\emptyset$.
 
Let $D^{\star}(k)=D(k,i_0)-D^{\star}(k,i_0)$.
Using similar techniques to those used already in this section we may prove that  
 \begin{eqnarray}
 \label{2019-04-19}
 p_1
 &
 :=
 &
 \Pr\{D(k,i_0)\ge 0.8{\bar D}(k,i_0),\,
 \forall k\in I\}=1-o(1),
 \\
 \label{2019-04-19+1}
 p_2
 &
 :=
 &
 \Pr
 \bigl\{
 D^{\star}(k)
 <
 0.3{\bar D}(k,i_0),
 \forall k\in I
 \bigr\}
 =1-o(1).
 \end{eqnarray} 
(\ref{2019-04-19}) and (\ref{2019-04-19+1}) imply {\bf P8c}.
 
It remains to prove (\ref{2019-04-19}) and (\ref{2019-04-19+1}).

Proof of (\ref{2019-04-19}). By (\ref{EqEDkiDk}) we have
$i_0^2\le{\bar D}(k,i_0)\le 1.1 \E D(k,i_0)$, 
$k\in I$.
Combining the union bound and Chebychev's inequality we obtain
\begin{eqnarray}
	1-p_1
	\le
	\sum_{k\in I}
	\Pr\bigl\{D(k,i_0)< 0.88\,\E D(k,i_0)\bigr\}
	%THIS SEEMINGLY SILLY 0.88 IS BY THE PURPOSE TO MAKE THE OMITTED STEPS VISIBLE 
	\le
	\sum_{k\in I}70\frac{\Var(D(k,i_0))}{(\E D(k,i_0))^2}
	=o(1).
\end{eqnarray}
In the last step we used 
$\E D(k,i_0)\ge i_0^2/1.1$ and invoked the  approximation
\begin{displaymath}
	\Var(D(k,i_0))
	=
	\bigl(\E D(k,i_0)\bigr)^2O(n^{-0.9})+O(n^{-1}\ln^3n))
	+
	\E D(k,i_0),
\end{displaymath}
which is shown using the same argument as in 
(\ref{2019-04-04+1}), (\ref{2019-04-03+7}), (\ref{EqDbisVariance}) above.

Proof of (\ref{2019-04-19+1}). 
We show below that
$\E D^{\star}(k)\le c'{\bar D}(k,i_0)\ln^{-3}n$. Then combining the union bound and Markov's inequality we obtain
\begin{eqnarray}
	1-p_2
	\le
	\sum_{k\in I}
	\Pr
	\bigl\{
	D^{\star}(k)
	\ge
	0.3{\bar D}(k,i_0)
	\bigr\}
	\le c''|I|\ln^{-3}n=o(1).
\end{eqnarray}
Given $k$ we upper bound $\E D^{\star}(k)$ by the expected number of vertex pairs $v\not=u$ such that $|\W'(v)|=|\W'(u)|=i_0$,
$\deg(v)=k$ and $\dist(u,v)\le \ln/(\ln\ln n)^3$.  The pairs with
different intersection sizes  
$|\W'(v)\cap \W'(u)|=r$ will be counted separately.

For $r=1$ the expected number of pairs is upper bounded by
\begin{eqnarray}
	\label{2019-04-19+3}
	&&\,
	n(n-1)
	\cdot
	{\binom{m}{i_0}}
	{\binom{n-2}{k-1}}
	{\binom{m-i_0}{i_0-1}}
	\cdot
	p^{i_0}(1-p)^{i_0(n-k-1)}p^*_{k,i_0}
	\\
	\nonumber
	&&
	\cdot
	\,
	\bigl(p(1-p)(1-(1-p)^{n-2})^{i_0-1}
	\cdot
	\bigl(
	(1-p)^2+2p(1-p)^{n-1}
	\bigr)^{m-2i_0+1}.
\end{eqnarray}
Here $n(n-1)$  counts ordered pairs 
$v\not=u$. ${\binom{m}{i_0}}
{\binom{n-2}{k-1}}{\binom{m-i_0}{i_0-1}}$
counts non intersecting subsets $A_{i_0}$, $A_{i_0-1}\subset W$ and $B_k\subset \V\setminus \{v\}$ with $u\in B_k$ that can realise  $\W'(v)$, $\W'(u)\setminus \W'(v)$ and ${\cal N}(v)$ respectively.
Furthermore, 
$p^*_{k,i_0}=\Pr\{{\cal A}_{A_{i_0},B_k}\}$ and 
$p^{i_0}(1-p)^{i_0(n-k-1)}$ is the probability that  all elements of $A_{i_0}$ are linked to $v$ and  none to $\V\setminus (B_k\cup\{v\})$. Next, 
$\bigl(p(1-p)(1-(1-p)^{n-2})^{i_0-1}$ is the probability that each element of $A_{i_0-1}$ is linked to $u$, none to $v$ and each  has more than one neighbor in ${\cal B}$. Finally, $\bigl(p(1-p)(1-(1-p)^{n-2})^{i_0-1}$
is the probability that none element of $\W\setminus(A_{i_0}\cup A_{i_0-1})$ belongs to 
$\W'(v)\cup \W'(u)$.

\noindent
Using (\ref{2019-03-28++1}),
(\ref{2019-03-28++5}), see also
(\ref{2019-03-28++9}), we show that  
(\ref{2019-04-19+3}) is at most
${\bar D}(k,i_0)
k
\frac{d_0^{i_0-1}}{(i_0-1)!}
e^{-d_0}(1+o(1))$.
% \end{document}
Next we bound $\frac{d_0^{i_0-1}}{(i_0-1)!}\le\frac{d_0^{i_0}}{i_0!}$ and use Stirling's approximation to $i_0!$. We have
\begin{equation}\label{2019-04-23-1}
	\frac{d_0^{i_0}}{i_0!}
	e^{-d_0}
	\le
	\Bigl(\frac{d_0}{i_0}\Bigr)^{i_0}
	e^{i_0-d_0}
	\le
	\Bigl(\frac{c}{c-1}\Bigr)^{1+(c-1)\ln n}
	e^{1-\ln n} 
	\le 
	c'
	\frac{\ln n}{n}
	\Bigl(\frac{c}{c-1}\Bigr)^{(c-1)\ln n}.
\end{equation}
In the last step we used $c/(c-1)=o(\ln n)$. Furthermore, for
$c>1$, $c=\Theta(1)$ there exists $\varepsilon>0$ such that $(c/(c-1))^{c-1}< e^{1-\varepsilon}$ uniformly in $n,m$ (because $x\to (1+x^{-1})^x$ increases for $x>0$ and approaches $e$ as 
$x\to +\infty$). Hence, the right side is bounded by $n^{-\varepsilon/2}$. 
We conclude that for $r=1$ the expected number of pairs is at most $c''{\bar D}(k,i_0)n^{-\varepsilon/3}$.

For $r=2$  we similarly upper bound  the expected number of pairs by
$c''{\bar D}(k,i_0)n^{-\varepsilon/3}$. 

For $r\ge 3$ the expected number of pairs is at most 
$n(n-1)p''_{-}
\le 
c'n^{-1}\ln^3n
\le 
c''{\bar D}(k,i_0)\ln^{-4}n$, for $k\in I$.

For $r=0$ we consider separately the pairs 
that are in the distance $\dist(u,v)=t\in \{2,3,\dots\}$. 
For $t=1$, the expected number of pairs 
is at most
\begin{eqnarray}
	\label{2019-04-23-2}
	&&\,
	n(n-1)
	{\binom{n-2}{k}}
	{\binom{m}{i_0}}
	{\binom{m-i_0}{i_0}}
	\cdot
	p^{i_0}(1-p)^{i_0(n-k-1)}p^*_{k,i_0}
	\\
	\nonumber
	&&
	\cdot
	\bigl(p(1-p)\bigr)^{i_0}
	\bigl((1-(1-p)^{n-2}\bigr)^{i_0-1}
	\bigl(
	(1-p)^2+2p(1-p)^{n-1}
	\bigr)^{m-2i_0}
	\cdot
	(mnp^2)^{t-2}ki_0p.
\end{eqnarray}
Here $n(n-1)$  counts ordered pairs 
$v\not=u$, ${\binom{n-2}{k}}$
counts  sets 
$B_k\subset \V\setminus \{u, v\}$ 
that realise  
${\cal N}(v)$, ${\binom{m}{i_0}}
{\binom{m-i_0}{i_0}}$ counts non-intersecting pairs $A_i,A'_i\subset \W$ that  realise $\W'(v)$, 
$\W'(u)$. Furthermore, $(mnp^2)^{t-2}ki_0p$ upper bounds the number of paths connecting $B_k$ with $A'_i$ and having 
$2t-1$ links. 
Proceeding as in (\ref{2019-04-19+3}), (\ref{2019-04-23-1})
we upper bound (\ref{2019-04-23-2}) by
$c''{\bar D}(k,i_0)
\frac{d_0^{d_0}}{i_0!}
e^{-d_0}\ln^tn$, 
where $\frac{d_0^{d_0}}{i_0!}e^{-d_0}
\le n^{-\varepsilon/2}$. Hence
(\ref{2019-04-23-2})
is at most $c''{\bar D}(k,i_0)n^{-\varepsilon/3}$.

\section{Probability of the first visit of a vertex}\label{SectionReturns-}

%As we mentioned above, here we use heuristics from 
%\cite{CoverTimeGiant} and the following papers.
Given a connected graph $G$ on the vertex set $\V$, let $\EW_u$ be the simple random walk starting from $u\in \V$. 
Let $T>0$ and $t\ge T$. In this section we will study the probability of the event $\A_{t}(v)$ that $\Walk_u$ does not visit $v$ in steps $T, T+1, \ldots, t$. For that we need to introduce some additional notions.    

Let $P_u^{(t)}(v) = \Pr\{\Walk_u(t) = v\}$, where $\EW_u(t)$ denotes the vertex visited at time $t=0,1,2,\dots$ (so that $P_u^{(0)}(u)=1$).
Assuming  that $G$ admits a stationary distribution 
$\pi=\{\pi_v, v\in\V\}$ (i.e., $\lim_{t\to+\infty}P_u^{(t)}(v) =\pi_v$, for all $u,v\in\V$) we have
$\pi_v=\deg(v)|{\cal E}(G)|^{-1}$. Given integer $T>0$
and $v\in \V$,
let
$$
R_{T,v}(z)=\sum_{j=0}^{T-1}
{\rm Pr}\{\Walk_v(j)=v\}z^t,
\qquad
z\in{\mathbb C}.
$$

In the following lemma we consider a sequence of connected graphs $\{G_n\}$, where $n$ is the number 
of vertexes of $G_n$. We assume that each graph admits a stationary distribution $\pi=\pi(n)$. 
Furthermore, we assume that  $T=T(n)$ is such that, for $t\ge T$
\begin{equation}\label{CooperFrieze1}
\max_{u,v\in\V}\Big|P_u^{(t)}(v)-\pi_v\Big|\le n^{-3}.
\end{equation}

The following lemma was proved in \cite{CooperFrieze2008}. It is stated there as Corollary~7.
\begin{lem}\label{LemNotVisit}
Suppose that $T=T(n)$ satisfies (\ref{CooperFrieze1}) and

\begin{itemize}
	\item[(i)] 
there exist $\Theta>0$, $C_0>0$ and $n_0>0$ such that uniformly in $n>n_0$ we have
$$
	\min_{|z|\le 1+(C_0T)^{-1}} |R_{T,v}(z)|\ge \Theta,
$$
	\item[(ii)] 
$T^2\pi_v=o(1)$ and $T\pi_v=\Omega(n^{-2}).$
\end{itemize}
Then there exists
\begin{equation}\label{CooperFrieze2}
p_v=\frac{\pi_v}{R_{T,v}(1)(1+O(T\pi_v))}
\end{equation}
such that for all $t\ge T$
\begin{equation}\label{CooperFrieze3}
{\rm Pr}\{\A_t(v)\}
=
\frac{1+O(T\pi_v)}{(1+p_v)^t}+ o(e^{-t/(2C_0T)}).
\end{equation}
\end{lem}
 We note that the bounds $O(T\pi_v)$ and $o(e^{-t/(2C_0T)})$ in (\ref{CooperFrieze2}) and (\ref{CooperFrieze3}) hold uniformly in $u,v$ and $n>n_0$, provided that  conditions (i), (ii) hold uniformly in  $v$ and $n>n_0$.

\noindent

\subsection{The expected number of returns}
%\label{SectionReturns}

\bigskip
Let $\G$ be an instance of the random intersection 
graph $\Gnmp$ and consider the simple random walk  
$\EW_v$ on $\G$ 
starting from $v\in \V$. We concentrate on $R_{T,v}(z)$ for such random walk (in order to determine $\PraG{\A_t(v)}$) and its variant which will be helpful in establishing $\PraG{\A_t(v)\cap \A_t(v')}$. 
Let $r_i=\PraG{\EW_v(i)=v}$ be the probability that the walk returns to $v$ at time $i$ (so that
$\EW_v(0)=v$  implies $r_0=1$).  We remark that 
$z\to R_{T,v}(z)=\sum_{i=0}^{T-1}r_iz^i$ is a random function depending on the realised graph ${\cal G}$. Furthermore  
given a pair of vertexes $x,y$  of ${\cal G}$ that  
 are in a distance at least~$20$, we denote by
${\cal G}_\varkappa$ the graph obtained from ${\cal G}$ by merging $x$ and $y$. 
Here $\varkappa=\{x,y\}$ represents  the new vertex 
obtained from the merged pair. We denote by 
$\deg(\varkappa)=\deg(x)+\deg(y)$ the degree of 
$\varkappa$.
In ${\cal G}_{\varkappa}$ we consider the simple random walk $\EW_{\varkappa}$ starting from 
$\varkappa$. 
Let $R_{T,\varkappa}(z)
=
R_{T,x,y}(z)=\sum_{i=0}^{T-1}{\bar r}_iz^i$, where 
${\bar r}_i=\PraG{\EW_{\varkappa}(i)=\varkappa}$.
 Furthermore, let $\tau_v$ ($\tau_{\varkappa}$) be the time of the first return of $\EW_v$ ($\EW_{\varkappa}$) to $v$ ($\varkappa$) in the interval $[1,+\infty)$.
 In the lemma below we assume that $m,n,p$ satisfy
conditions of Theorem  \ref{ThmMain}.

\begin{lem} 
 Let $C_0>0$. Assume that $T=T(n,m)\to\infty$ and $T=o(\ln^3 n)$.
We have {\whp} 
%uniformly in $v\in V$
\begin{eqnarray}\label{R_T}
&&
\sup_{|z|\le 1+(C_0T)^{-1}}
|R_{T,v}(z)|
=
1+O(\ln^{-1}n)
\quad
\forall \ v\in V,
\\
\label{R_T+}
&&
\sup_{|z|\le 1+(C_0T)^{-1}}
|R_{T,x,y}(z)|
=
1+O(\ln^{-1}n)
\quad
\forall  x,y\in V
\ \
{\rm {with}\ }
\ \  
\dist(x,y)\ge 20,
\\
\label{R_T+1}
&&
R_{T,v}(1)=1+{\bar p}_v+O(\ln^{-2}n),
\quad
\
{\text{where}}
\quad
\
{\bar p}_v=\PraG{\tau_v\le T-1}\asymp
\ln^{-1}n,
\\
\label{R_T+2}
&&
R_{T,x,y}(1)=1+{\bar p}_{\varkappa}+O(\ln^{-2}n),
\quad
{\text{where}}
\quad
{\bar p}_{\varkappa}=\PraG{\tau_{\varkappa}\le T-1}\asymp
\ln^{-1}n,
\\
\label{R_T+3}
&&
\frac{\deg(\varkappa)}{1+{\bar p}_{\varkappa}}
=
\frac{\deg(x)}{1+{\bar p}_x}
+
\frac{\deg(y)}{1+{\bar p}_y}
+
O(\deg(\varkappa)\ln^{-2}n).
\end{eqnarray}
Furthermore, (\ref{R_T}), (\ref{R_T+1}), respectively,
(\ref{R_T+}), (\ref{R_T+2}), (\ref{R_T+3}) hold uniformly in
$v\in \V$, respectively, uniformly in $x,y\in \V$  satisfying
$\dist(x,y)\ge 20$.
\end{lem}
\begin{proof}
We establish (\ref{R_T}-\ref{R_T+3}) for 
${\cal G}$ having properties $\bf P1$-$\bf P8$, see 
Lemma \ref{LemProperties}. 

{\it Proof of} (\ref{R_T}). We have 
$\bigl||R_{T,v}(z)|-1\bigr|
\le 
|\sum_{i=1}^{T-1}r_iz^i|$
 and for $|z|\le 1+(C_0T)^{-1}$
\begin{equation}\label{2019-03-11+6}
\Bigl|\sum_{i=1}^{T-1}r_iz^i\Bigr|
\le 
\sum_{i=1}^{T-1}r_i|z|^i
\le 
(1+(C_0T)^{-1})^T{\tilde R}
\le e^{C_0^{-1}}{\tilde R},
\qquad 
{\tilde R}:=\sum_{i=1}^{T}r_i.
\end{equation}

We show  below that ${\tilde R}=O(\ln^{-1}n)+O(T\ln^{-5} n)$  uniformly in 
$v\in V$.  
Note that 
${\tilde R}$ is the expected number of returns 
to $v$ of the random 
walk $\EW_v$ in the time interval 
$[1, T]$.

We begin with an observation, denoted (O), about random walks
on  directed graph with the vertex set 
$\{0,{\onul}, 1,2,3\}$, where $3$ is an 
absorbing state. 
\medskip

\noindent {\bf (O)} {\it Assume, that
the transitional probabilities
$p_{0,3}=p_{\onul,3}=p_{1,3}=p_{0,2}=0$
and 
$0<p_{0,\onul}, p_{0,1}, p_{1,2}, p_{2,1}, p_{2,3}<1$ are fixed.
The walk starts at $0$ and it is allowed to make 
$t$ steps.
 Then for any $t$, the expected number of returns to $0$ before 
 visiting
 $2$ is maximized if we choose
\begin{displaymath}
p_{\onul,0}=1,
\
\
p_{\onul,1}=p_{\onul,2}
=
p_{1,\onul}=p_{2,\onul}=0,
\
\
p_{1,0}=1-p_{1,2},
\
\
p_{2,0}=1-p_{2,1}-p_{2,3}.
\end{displaymath}}

{\it Case} (1). Assume that 
$\N_1(v)\cup\cdots\cup \N_7(v)$ contains no small vertexes.
The random walk  $\EW'(i)=\min\{7, \dist(v, \EW_v(i)\}$
moves along the path of length $7$ and has the state
 space $\{0,1,\dots, 7\}$. Its
transitional probabilities  satisfy inequalities
\begin{equation}\label{2019-03-11}
p'_{j+1,j}\le c'/\ln n,
\qquad
p'_{j,j}\le c'/\ln\ln n,
\quad
 0\le j\le 6,
 \end{equation} 
 where $c'$ is an absolute constant.
Indeed, by {\bf P6}, every $u\in {\cal N}_{j+1}(v)$ is 
adjacent to at most two vertexes from 
${\cal N}_{j}(v)$. Now {\bf P5} implies 
$p'_{j+1,j}=O(\ln^{-1}n)$. Furthermore, by {\bf P4}, 
each of these vertexes shares with $u$ at most 
$a_{\star}(\ln n)/\ln\ln n$ common neighbors from 
${\cal N}_{j+1}(v)$.
In addition, by {\bf P7}, there can be at most $2$ 
vertexes in ${\cal N}_{j+1}(v)$ adjacent to $u$ and having no common neighbors with $u$ located in 
${\cal N}_i(v)$.
 Therefore, every $u\in {\cal N}_{j+1}(v)$ can have at 
 most $2+2a_{\star}(\ln n)/\ln\ln n$ neighbors in 
${\cal N}_{j+1}(v)$ altogether.
Now {\bf P5} imply 
$p'_{j+1,j+1}=
O(1/\ln\ln n)$, for $0\le j\le 5$.
Finally, we obviously have $p'_{0,1}=1$.
 \begin{figure}
	\includegraphics[width=14cm]{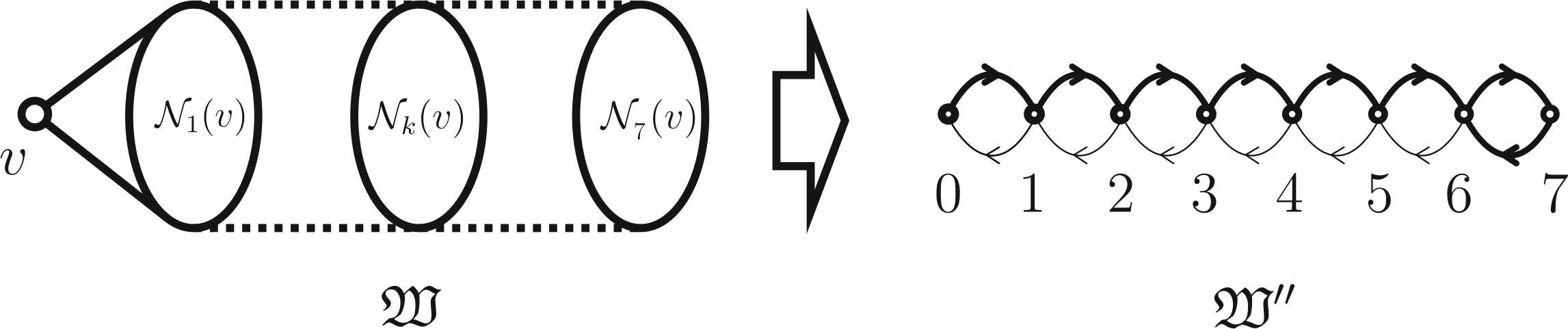}
	\caption{Transition from walk $\EW$ to $\EW''$ in Case (1). }
	%\label{FigWalk1}
\end{figure}

%
%Indeed, every $u\in N_{j+1}$ is large and, by Fact 2, 
%it 
%has at least $\Theta(1)\ln n$ neighbors. By (O2), at %most two 
%of its neighbors belong to $N_{j}$. By (O1) 
%at most $O((\ln\ln n)^{-1}$ belong to $N_{j+1}$, for %$0\le j\le 5$.

The random walk $\EW'$ is lazy: it may stay at state $j>1$ for several consecutive steps. Let $\EW''$
be the fast random walk defined by $\EW'$  as follows: 
$\EW''$  only makes a step when $\EW'$ changes its state. In the latter case the moves of $\EW'$ and $\EW''$ coincide.  Its transitional probabilities
\begin{equation}\label{2019-03-11+2}
p''_{0,1}=p''_{7,6}=1,
\quad
p''_{j+1,j}=p'_{j+1,j}/(1-p'_{j,j})
\quad
{\text{ for}}
\quad
0\le j\le 5
\quad
{\text{ and}}
\quad
p''_{i,i}=0
\quad
\forall i.
\end{equation} 
We have
 ${\tilde R}=R'\le  R''$, where
 $R'$ and  
$R''$ denote the expected numbers of 
returns  to $0$ within the first 
$T$
 steps of respective random walks $\EW'$ and $\EW''$.
 We split $R''=R''_1+R''_2$, where $R''_1$ is the expected number of returns to $0$ before the first visit of $7$.
% $\EW^*$ visits $j=7$ at first time$i_0$. 
We have $R''_1\le \E_{\cal G} X$, where $X$ is the number of backward steps made by $\EW''$ before visiting $7$.
The inequality $p''_{j+1,j}\le c'(1+o(1))/\ln n$, see (\ref{2019-03-11}), (\ref{2019-03-11+2}), implies 
\begin{displaymath}
{\rm Pr}_{\cal G}\{X=k\}
\le 
((6c'+o(1))/\ln n)^k, 
\quad
k\ge 0.
\end{displaymath}
Hence $\E_{\cal G} X=O(\ln^{-1}n)$ and we obtain 
$R''_1=O(\ln^{-1}n)$.
Furthermore,  
after visiting $7$ the random walk $\EW''$ moves to $6$.
Starting from $6$ the walk may visit $0$ before visiting $7$ again, we call such event  
a success.  The probability of success is $O(\ln^{-6}n)$ see, e.g., formula (30) in \cite{CooperFrieze2008}.
 The expected number of successes within the first $T$ steps of the random walk  
is at most $O(T\ln^{-6}n)$.
Hence $R''_2$, the expected number of returns 
to $0$ after the first  visit of $7$, is  at most
$O\bigl(T(1+ R''_1)\ln^{-6}n\bigr)$. 
Here $R''_1$ accounts for the returns to $0$ 
after a success and before visiting $7$ again.
We conclude that 
${\tilde R}\le R''_1+ R''_2=O(\ln^{-1}n)+O(T\ln^{-6}n)$.
\smallskip

{\it Case} (2). Assume that 
$\N_{k+1}(v)$ contains a small vertex, 
say ${\overline v}$, for some 
$0\le k\le 5$.
Note that, by {\bf P7}, 
there is no other small vertex in ${\cal G}$
within the 
distance $O(\ln\ln n)$ from $v$.
%We 
%split 
%${\tilde R}={\tilde R}_1+{\tilde R}_2$, %where 
%${\tilde R}_1$ is the expected number of %returns to $v$  
%by $\EW_v$ in the time interval $[1;T]$ and %before the 
%walk enters the set $N_7$ at the first time.  
%
%We show that${\tilde R}_1=O(\ln^{-1}n)$. 
Now we define 
$\EW'(i)=\min\{7, \dist(v, \EW_v(i))\}$, 
for $\EW_v(i)\not=\ov$, and put
$\EW'(i)=\ok$, for 
$\EW_v(i)=\ov$. It is a random walk on the state space 
$S_k=\{0,1,\dots, k,\ok,k+1,\dots, 7\}$. 
Furthermore, let $\EW''$ be  
the corresponding fast random walk on $S_k$: 
$\EW''$
only makes a step when $\EW'$ changes its 
state and  in the latter case the moves of 
$\EW'$ and $\EW''$ coincide. 
Arguing as in  
(\ref{2019-03-11}), (\ref{2019-03-11+2}) we obtain the 
corresponding inequalities for the transitional 
probabilities ${\overline p}_{i,j}$ of $\EW''$
\begin{equation}\label{1MB}
\op_{k,\ok},
\,
\op_{r, \ok},
\,
\op_{j,j-1}
\le 
c'/\ln n, 
\
\
\
r=k+1,k+2,
\
\
1\le j\le 6.
\end{equation}
Note that $\op_{\ok,k}, \op_{k,\ok}, 
\op_{j,j+1}>0$, $0\le j\le 6$. Furthermore, we have
 $\op_{\ok,r}, \op_{r,\ok}>0$ 
whenever $\ov$ has a neighbor in $\N_{r}(v)$, 
$r=k+1,k+2$. 
Moreover,  we have $\op_{7,6}=1$ and $\op_{j,j}=0$ 
for all $j\in S_k$.
Finally, 
$\op_{0,1}=1$ for $k>0$ and 
$\op_{0,1}+\op_{0,{\bar 1}}=1$ for $k=0$.
All the other transitional probabilities 
$\op_{i,j}$ are zero. 
From now on 
we consider the cases $k=0$ and $k\ge 1$ separately.

\begin{figure}
	\includegraphics[width=14cm]{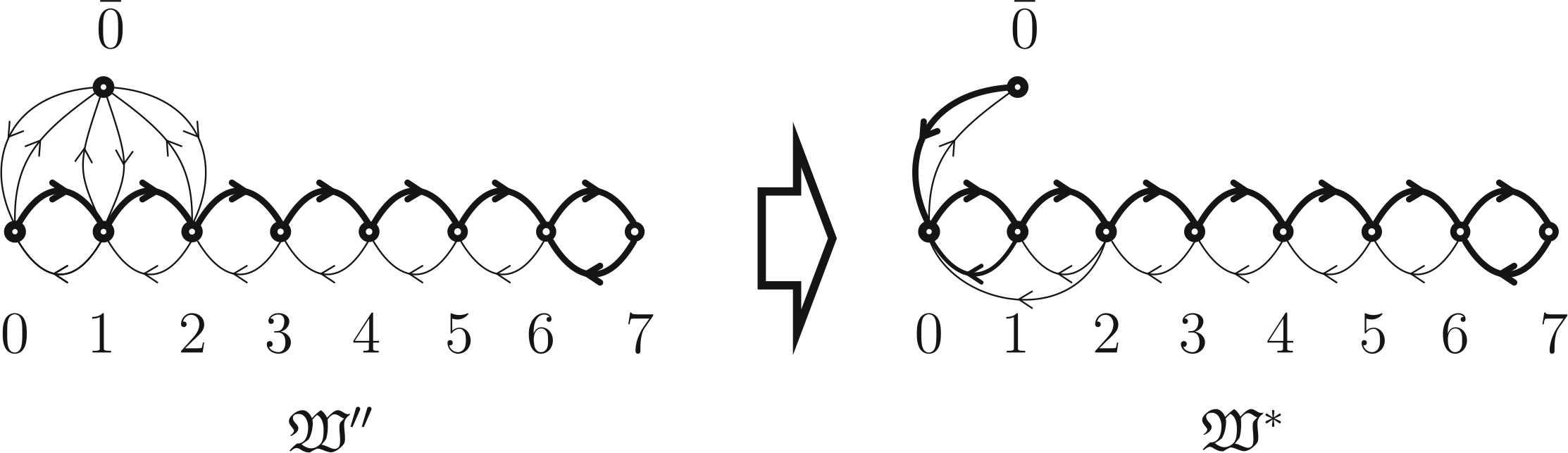}
	\caption{Random walks $\EW''$ and $\EW^*$ in Case (2) with $k=0$.}
	%\label{FigWalk2}
\end{figure}

Assume that  $k=0$, i.e.,  $\ov\in \N_1(v)$.
Let  $\EW^*$ be the
 random walk on $S_0$ starting from $0$ and with transitional probabilities $p^{\star}_{i,j}=\op_{i,j}$ for each $(i,j)\in  S_0\times S_0$, but
 \begin{equation}\label{2019-03-11+4}
 p^*_{{\overline 0},0}=1,
 \
\ 
 p^*_{{\overline 0},r}
 =
 p^*_{r,{\overline 0}}
 =
 0,
 \
 \
p^*_{r,0}=\op_{r,0}+\op_{r,{\overline 0}},
\quad
r=1, 2.
\end{equation} 
We have ${\tilde R}=R'\le R''\le R^*$,
where $R'$, $R''$ and  $R^*$ denote the expected numbers of returns to $0$ within the first 
$T$ steps of respective random walks 
$\EW'$, 
$\EW''$ and
$\EW^*$. The last inequality follows from observation (O). Let us consider the first $T$ steps of $\EW^*$.
We split $R^*=R^*_2+R^*_3$, where
$R^*_2$ ($R^*_3$)  denotes the 
expected number of returns
to $0$ 
before (after) the first visit to $2$.
From  (\ref{1MB}), (\ref{2019-03-11+4}) we easily obtain
that  $R^*_2=O(\ln^{-1}n)$.
After visiting $2$ the walk $\EW^*$ moves to $3$  with probability at least 
$1-2c'/\ln n$ and it moves towards $0$ with probability 
at most $2c'/\ln n$. In the latter case the random walk 
will  be back at $2$ after perhaps visiting $0$ and the 
expected number of  visits to $0$ before returning to 
$2$  is at most $1+R^*_2$.  Hence, the expected number of returns to $0$ after visiting $2$ and 
 before visiting $3$ is 
at most
$O((1+R^*_2)\ln^{-1} n)=O(\ln^{-1}n)$.
  Next we 
consider random walk $\EW^*$ 
restricted to the path $\{2,3,\dots, 7\}$, where
  $2$ and $7$ are reflecting states.
Assuming that the walk starts at $2$ we add  the expected number of at most $O(\ln^{-1}n)$  visits to $0$ after each return to $2$.
Proceeding as in {\it Case} (1) we estimate $R^*_3\le
O(\ln^{-1}n)+O(T\ln^{-5}n)$

\begin{figure}
	\includegraphics[width=14cm]{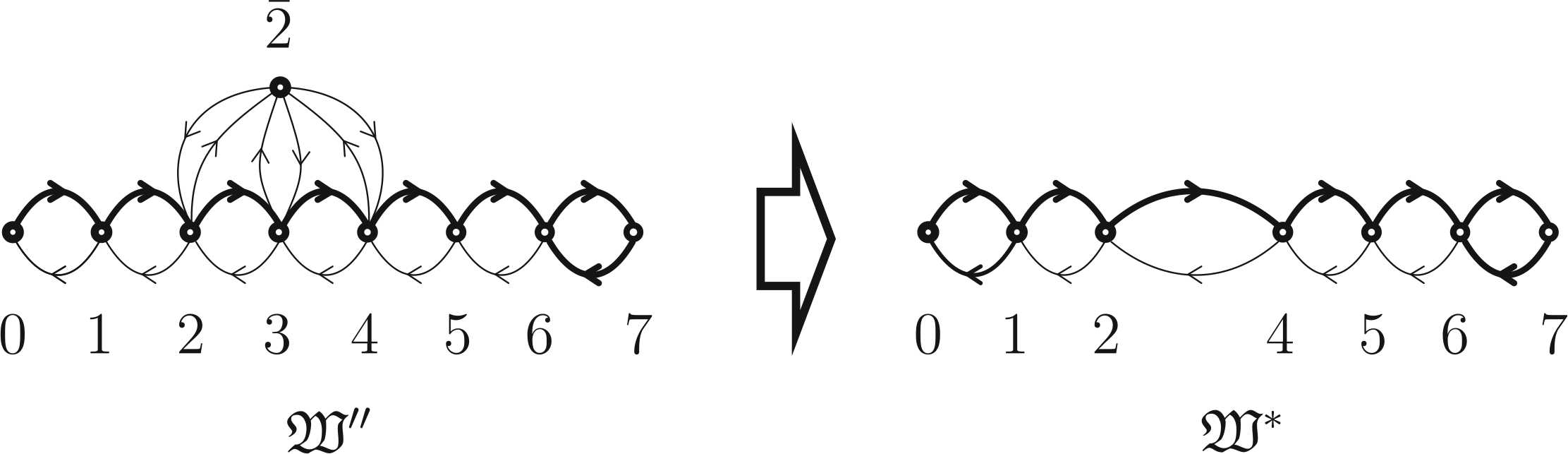}
	\caption{Random walks in Case (2) with $k=2$.}
%\label{FigWalk3}
\end{figure}
\smallskip
Assume that $1\le k\le 5$. We  follow the movements of 
$\EW''$ 
on the subset  
$S_k^{*}=\{0,1,\dots, k , k+2,\dots 7\}\subset S_k$ and
only register a move when the walk changes its state
in $S_k^{*}$.  
The walk moves along the path $S_k^{*}$ and has left 
and right reflecting states $0$ and $7$.
From (\ref{1MB})
we obtain that  it
 moves right (from each state but $7$) with probability at least 
$1-2c'\ln^{-1}n$. Arguing as in {\it Case} (1) we show that the expected number of returns to $0$ within the first $T$ steps is 
$O(\ln n)+O(T\ln^{-5}n)$. Obviously, it is an upper bound for ${\tilde R}$.

{\it Case }(3). For  $k=6$, i.e., ${\ov}\in \N_7(v)$, the set
$\N_1(v)\cup\cdots\cup \N_6(v)$ contains no small vertex. The argument used in Case (1) yields the bound 
${\tilde R}=O(\ln n^{-1})+O(T\ln^{-5}n)$.

\smallskip

{\it Proof of} (\ref{R_T+}). We proceed as in  (\ref{2019-03-11+6}). Using the fact that
${\bar R}:=\sum_{i=1}^{T}{\bar r}_i$
is the expected number of returns 
to $\varkappa$ of the random 
walk $\EW_{\varkappa}$ in the time interval $[1; T]$
we show that ${\bar R}=O(\ln^{-1}n)+O(T\ln^{-5} n)$.

Note that the vertex sets
${\cal N}_x:={\cal N}_1(x)\cup\cdots\cup {\cal N}_9(x)$ 
and
${\cal N}_y:={\cal N}_1(y)\cup\cdots\cup {\cal N}_9(y)$,
defined by ${\cal G}$, 
do not intersect since $x$ and $y$ are at distance at 
least $20$. 
We paint vertexes of ${\cal N}_x$ red and those of 
${\cal N}_y$ blue. While $\EW_{\varkappa}$ stays in 
${\cal N}_x$ (respectively ${\cal N}_y$) we call 
$\EW_{\varkappa}$
red (respectively blue). The path drawn by  red random 
walk 
corresponds to that of $\EW_x$ in ${\cal N}_x$ and the 
path drawn by blue random walk corresponds to that of  
$\EW_y$ in ${\cal N}_y$ (the walk may change its color after every visit to $\varkappa$).
For $1\le i\le 9$ we denote by
${\cal N}_i(\varkappa)={\cal N}_i(x)\cup{\cal N}_i(y)$ 
the set of vertexes at distance $i$ from 
$\varkappa$ in ${\cal G}_{\varkappa}$.
Note that 
${\cal N}_x\cup{\cal N}_y$ may contain at most two small 
vertexes, by $\bf P7$. Therefore, at least one of the 
sets ${\cal N}_i(\varkappa)$, $i=7,8,9$ has no small 
vertexes.  Assume it is ${\cal N}_7(\varkappa)$ (the 
cases $i=8,9$ are treated in the same way). 

Now we  analyse the blue and red walks similarly as in the proof of  (\ref{R_T}) above.
At the 
moment of the first visit of  
 ${\cal N}_7(\varkappa)$ by $\EW_{\varkappa}$ the 
 expected number of returns to $\varkappa$  is 
 $O(\ln^{-1} n)$.
Indeed, the number of returns is the sum of returns of the red and the blue walks. 
But the expected number of returns of the red walk 
 before it reaches ${\cal N}_7(x)$ 
 is the same as that of $\EW_x$ in ${\cal G}$. This number is $O(\ln^{-1} n)$, see {\it Cases} (1), (2) above. Similarly, the expected number of returns  of the blue walk before it reaches 
 ${\cal N}_7(y)$ is $O(\ln^{-1}n)$. 
After the first visit of ${\cal N}_7(\varkappa)$ the random walk $\EW_{\varkappa}$ stays in a distance at least $7$ from $\varkappa$ until it makes the first move from ${\cal N}_7(\varkappa)$ to 
${\cal N}_6(\varkappa)$. This move can be red or blue. A red (blue) move is successful if continuing from 
${\cal N}_6(x)$ (${\cal N}_6(y)$) the red (blue) walk visits 
$\varkappa$ before visiting 
${\cal N}_7(\varkappa)$ again. 
The probability of success is $O(\ln^{-5}n)$, see {\it Cases} (1), (2) above. Note that after a successful visit to 
$\varkappa$ and before visiting 
${\cal N}_7(\varkappa)$
again the walk $\EW_{\varkappa}$ may return several times to $\varkappa$, but the expected number of such returns is $O(\ln^{-1}n)$. Hence, each success occurs with probability $O(\ln^{-5}n)$ and it adds $1+O(\ln^{-1}n)$ expected number of returns to 
$\varkappa$. The expected number of successes in the time interval $[1,T]$ is at most $O(T\ln^{-5}n)$. 
Therefore 
${\bar R}= O(\ln^{-1}n)+O(T\ln^{-5}n)$.

{\it Proof of} (\ref{R_T+1}), (\ref{R_T+2}). 
We only show (\ref{R_T+1}). The proof of (\ref{R_T+2}) is much the same.
Let $Z_v(t)$ 
%and $Z_{\varkappa}(t)$ 
be the number of returns of  $\EW_v$ 
%and $\EW_{\varkappa}$ 
to $v$ 
%and $\varkappa$ 
in the time interval $[t,T-1]$. We define $Z_v(t)\equiv 0$ for $t\ge T$.  Let 
${\mathbb I}_{\{\tau_v\le T-1\}}$ be the indicator of the event $\tau_v\le T-1$.
%Recall that $\varkappa=\{x,y\}$.
We have
%\begin{eqnarray}
%\nonumber
%&&
%R_{T,v}(1)=1+\E Z_v(T), 
%\qquad
%\quad
%R_{T,\varkappa}=1+\E Z_{\varkappa}(T),
%\\
%\nonumber
%&&
%\E Z_v(T)={\tilde R}=O(\ln^{-1}n),
%\qquad
%\E Z_{\varkappa}(T)=R'=O(\ln^{-1}n),
%\end{eqnarray} 
\begin{displaymath}
R_{T,v}(1)=1+\E_{\cal G} Z_v(1) 
\qquad
{\rm{and}}
\qquad
\E_{\cal G} Z_v(1)\le {\tilde R}=O(\ln^{-1}n).
\end{displaymath} 
The last bound is shown in the proof of (\ref{R_T}) above.
 Note that ${\mathbb I}_{\{\tau_v\le T-1\}}\le
Z_v(1)$ implies  ${\bar p}_v\le \E_{\cal G} Z_v(1)=O(\ln^{-1}n)$.
Furthermore, {\bf P3}  implies $\Pr_{\cal G}(\tau_v=2)\ge 1/{\Delta}$ and (\ref{2019-02-25+1}) implies 
$1/{\Delta}\asymp \ln^{-1}n$. Now from the inequality
${\bar p}_v
\ge
\Pr_{\cal G}(\tau_v=2)$ 
we obtain 
${\bar p}_v\asymp \ln^{-1}n$.

Let us prove the first relation of (\ref{R_T+1}).
The identity
\begin{displaymath}
Z_v(1)={\mathbb I}_{\{\tau_v\le T-1\}}Z_v(1)
=
{\mathbb I}_{\{\tau_v\le T-1\}}(1+Z_v(\tau_v+1))
\end{displaymath}
implies $\E_{\cal G} Z_v(1)={\bar p}_v+{\hat R}$, where
${\hat R}
=
\E_{\cal G} {\mathbb I}_{\{\tau_v\le T-1\}}
Z_v(\tau_v+1)=O(\ln^{-2}n)$. Indeed,
\begin{displaymath}
{\hat R}
=
\sum_{i=1}^{T-1}\E_{\cal G} {\mathbb I}_{\{\tau_v=i\}}Z_v(i+1) 
=
\sum_{i=1}^{T-1}
{\rm Pr}_{\cal G}\{\tau_v=i\}
\E_{\cal G} Z_v(i+1)
\le
{\bar p}_v\E_{\cal G} Z_v(1)=O(\ln^{-2}n).
\end{displaymath}
Here we used $\E_{\cal G}Z_v(i)\le \E_{\cal G} Z_v(1)$, for $i\ge 1$.

{\it Proof of} (\ref{R_T+3}).
Let $q_x=\deg(x)/\deg(\varkappa)$ and
$q_y=\deg(y)/\deg(\varkappa)$ be the 
probabilities that the first move of $\EW_{\varkappa}$ is red and blue respectively.
Let $A_{x}$  be the event that the first return 
of $\EW_x$ in the time interval $[1,T-1]$ 
occurs before the first visit to ${\cal N}_7(x)$.
Similarly we define the events $A_y$ and 
$A_{\varkappa}$. Define the probabilities 
${\bar p}'_u
=
\Pr_{\cal G}\bigl\{\{\tau_u\le T-1\}\cap A_u\bigr\}$ 
for $u=x,y,\varkappa$.
The relations  
$\Pr_{\cal G}\{A_u\}=1-O(T\ln^{-5}n)$ 
(which are shown using the same argument as  (\ref{R_T}), 
(\ref{R_T+})) imply
${\bar p}_u={\bar p}_u'+O(T\ln^{-5}n)$  
for $u=x,y,\varkappa$. Now the identity 
${\bar p}'_{\varkappa}
= 
q_x{\bar p}'_{x}
+
q_y{\bar p}'_{y}
$
implies 
${\bar p}_{\varkappa}
=
q_x{\bar p}_{x}
+
q_y{\bar p}_{y}+O(T\ln^{-5}n)$.  
Combining this identity with relations 
\begin{displaymath}
(1+{\bar p}_u)^{-1}
=
1-{\bar p}_u+O({\bar p}_u^2)=1-{\bar p}_u+O(\ln^{-2}n),
\quad
u=x,y,\varkappa
\end{displaymath}
we obtain
\begin{displaymath}
\frac{1}{1+{\bar p}_\varkappa}
=
\frac{q_x}{1+{\bar p}_x}
+
\frac{q_y}{1+{\bar p}_y}
+O(T\ln^{-5}n)+O(\ln^{-2}n).
\end{displaymath}
\end{proof}

\subsection{Probability of the first visit by time (\ref{2019-02-21})}

%Here we define a set of special 
%vertexes $\Set$ which are visited by the random 
%walk in the last phase before covering all the 
%vertexes.  Then we evaluate the probability that given $v\in \Set$ is not visited by the  time (\ref{2019-02-21}). 

Here for any
typical $\G$ and any  $u,v\in \V$
we
evaluate the probability 
that     
the simple random walk  $\EW_u$ 
starting from  $u\in \V$
does not visit $v$ after time 
$T$  and before time (\ref{2019-02-21}).  
We choose $T=\Theta(\ln n)$ satisfying
(\ref{CooperFrieze1}) and Lemma \ref{LemNotVisit} (ii). By
 {\bf P1}, {\bf P2}, {\bf P3}, such 
 $T=\Theta(\ln n)$  exists and it does not depend on particular instance 
 $\G$, see \cite{SincalirJerrum}.

\noindent
 We write the principal term of (\ref{2019-02-21}) in the form 
$\lambda mn^2p^2$ and
approximate it by 
\begin{equation}\label{EqDeft}
t_0=\lambda_0 mn^2p^2
\qquad 
\text{and}
\qquad 
t_1=\lambda_1mn^2p^2.
\end{equation}
Here
\begin{equation}\label{EqDefLambda}
\lambda=\ln\frac{np}{\ln\left(a+1\right)},
\qquad
\lambda_1=\ln\frac{np}{\ln\left(Aa+1\right)},
\qquad
\quad
\lambda_0=(1+\eps_n)\lambda
\end{equation}
and
\begin{equation*}
A=\exp\left({\frac{10\ln\ln n}{(c-1)\ln n}}\right),
\qquad
a=\frac{c-1}{c}(e^{np}-1),
\qquad
\eps_n=\frac{\ln\ln n}{\ln n}.
\end{equation*}
In Fact 5 we collect several 
observations about $\lambda$'s. The proof is given in Appendix B.

\begin{fact}\label{FactLambda}\ 
	\begin{itemize}
		\item [(i)] $\lambda>0$ is bounded away from $0$
		by a constant.
		\item [(ii)] $\lambda\le 2\ln\ln n$.
		\item [(iii)] If $(c-1)\le (\ln n)^{-1/3}$ then $\lambda\ge (\ln\ln n)/4$.
		\item [(iv)] $\lambda=\Theta\bigl(1+|\ln (c-1)|\bigr)$.
		\item [(v)] $\lambda_1=(1+o(1))\lambda$.
		\item [(vi)] 
		$\ln(c\kappa/(c\kappa-1))<\lambda<\ln(c/(c-1))$,
		\quad
		where
		\quad 
		$\kappa=np(1-e^{-np})^{-1}$.
	\end{itemize}	
\end{fact} 
\noindent
Note that inequalities $\lambda_1<\lambda<\lambda_0$ (the first one follows as $A>1$) and Fact 5 (v) imply
\begin{displaymath}
t_1<\lambda m n^2p^2<t_0
\qquad
{\text{and}}
\qquad
 t_1= (1-o(1))t_0
\asymp
\bigl(1+|\ln(c-1)|\bigr)n\ln n.
 \end{displaymath}
%Introduce the sets  
%\begin{displaymath}
%I_1
%=
%\{k:i_0\le k\le k_0\text{ and } \EDkiO > k_0^{-2}e^{k\lambda_1}\},
%\qquad
%I_0=\{k:i_0\le k\le k_0\}\setminus I_1
%\end{displaymath}
%(recall that $i_0, k_0$ and $\EDkiO$ are defined in 
%{\bf P8})
%and define the set of special vertexes 
%\begin{equation*}
%\Set
%=
%\Bigl\{
%v\in\V: |\W'(v)|=i_0,\, \deg(v)\in I_1,
%\,
%\min_{v':|\W'(v')|=i_0, v'\not=v}
%\dist(v,v')
%\ge 
%\frac{\ln n}{(\ln\ln n)^3} 
%\Bigr\}.
%\end{equation*}

Now we show  that whp we have  uniformly in 
$v,x,y\in\V$ 
with $\dist(x,y)\ge 20$
\begin{eqnarray}
\label{A++}
\PraG{{\cal A}_{t_i}(v)}
&
=
&
e^{-\deg(v)\lambda_i/(1+{\bar p}_v)}\bigl(1+o(1)\bigr)+o(n^{-3})
\\
\nonumber
&
\ge
&
e^{-\deg(v)\lambda_i}\bigl(1+o(1)\bigr)+o(n^{-3}),
\\
\label{2019-04-26}
\PraG{\A_{t_1}(x)\cap \A_{t_1}(y)}
&
= 
&
\PraG{\A_{t_1}(x)}\PraG{\A_{t_1}(y)}
(1+o(1))
+o(n^{-3}).
\end{eqnarray}
We recall that ${\bar p}_v=O(\ln n)$ is defined in \eqref{R_T+2}.

\begin{proof}[Proof of \eqref{A++}]
From {\bf P1}, {\bf P3} and  (\ref{CooperFrieze2}), (\ref{R_T+1}) 
we obtain for $T=O(\ln n)$ that
\begin{eqnarray}
\label{XI-28-2}
&&
\pi_v
=
\deg(v)/\bigl(2|{\cal E}(\G)|\bigr)
=
\deg(v)(mn^2p^2)^{-1}(1+O(n^{-1/2}))
=
O(n^{-1}),
\\
\label{XI-28-1}
&&
p_v
=
\pi_v(1+{\bar p}_v)^{-1}(1+O(\ln^{-2}n))
=
\pi_v(1+O(\ln^{-1}n)).
\end{eqnarray} 
Combining these relations and using {\bf P3} and  Fact~\ref{FactLambda} (ii) we  obtain 
\begin{equation}\label{XI-28-3}
t_ip_v
=
\frac{\lambda_i\deg(v)}{1+{\bar p}_v}
(1+O(\ln^{-2}n))
=
\frac{\lambda_i\deg(v)}{1+{\bar p}_v}
+
O\Bigl(\frac{\ln\ln n}{\ln n}\Bigr).
\end{equation}
Furthermore, (\ref{XI-28-2}), 
(\ref{XI-28-1}) imply
\begin{equation}\label{atsibodo}
(1+p_v)^{t_i}
=
e^{t_i\ln(1+p_v)}
=
e^{t_ip_v+O(t_ip_v^2)}
=
e^{t_ip_v}\bigl(1+O(t_i/n^2))\bigr).
\end{equation}
Finally,  we apply Lemma \ref{LemNotVisit} 
with $C_0=1$ (condition (i) of Lemma \ref{LemNotVisit}
holds by (\ref{R_T})) and derive (\ref{A++})  
from (\ref{CooperFrieze3}), (\ref{XI-28-3}) and (\ref{atsibodo}). Note that in this step we estimate the remainder term of (\ref{CooperFrieze3}),
$e^{-t_1/(2C_0T)}=o(n^{-3})$.
\end{proof}

\begin{proof}[Proof of \eqref{2019-04-26}]
We use the observation  of 
\cite{CooperFrieze2008}
that ${\cal A}_{t_1}(x)\cap {\cal A}_{t_1}(y)=
{\cal A}_{t_1}(\varkappa)$, where 
${\cal A}_{t_1}(\varkappa)$ is the event that the same 
random walk $\EW_u$, when considered in 
${\cal G}_{\varkappa}$, does not visit the vertex 
$\varkappa$ of ${\cal G}_{\varkappa}$ in steps 
$T,T+1,\dots, t_1$. We recall that 
${\cal G}_{\varkappa}$ is obtained from ${\cal G}$ by merging
 vertexes $x$ and $y$ into one vertex denoted 
$\varkappa$.
Therefore, 
$\PraG{{\cal A}_{t_1}(x)\cap {\cal A}_{t_1}(y)}
=
\PraG{{\cal A}_{t_1}(\varkappa)}$.
Proceeding as in the proof of  (\ref{A++})  
(see also remark below)  we show that 
\begin{equation}
\label{A-lambda_1-varkappa}
\PraG{\A_{t_1}(\varkappa)}= 
e^{-\lambda_1\deg(\varkappa)/(1+{\bar p}_{\varkappa})}
(1+o(1))+o(n^{-3}).
\end{equation}
Furthermore, combining  (\ref{A++}) 
with
(\ref{A-lambda_1-varkappa}) and using 
(\ref{R_T+3})  
we obtain
$$
\PraG{\A_{t_1}(\varkappa)}
= 
\PraG{\A_{t_1}(x)}\PraG{\A_{t_1}(y)}
(1+o(1))
+
o(n^{-3})
$$
thus showing \eqref{2019-04-26}.
We remark that (\ref{A-lambda_1-varkappa}) refers to
the random walk in ${\cal G}_{\varkappa}$ starting at 
$u$. We note that the expansion property {\bf P2} 
extends to ${\cal G}_{\varkappa}$ and, therefore, Lemma  
\ref{LemNotVisit} applies with the same (mixing time) 
$T$. The only difference is that now we verify 
condition (i) of in Lemma \ref{LemNotVisit} using  
(\ref{R_T+})
instead of (\ref{R_T}).
\end{proof}

\noindent
Finally, we note that
${\bar p}_v=O(\ln^{-1} n)$
implies 
$\lambda_0/(1+{\bar p}_v)\ge \lambda$.  
Thus  by \eqref{XI-28-3}, \eqref{atsibodo} we get
\begin{equation}\label{EqAt0upper}
(1+p_v)^{-t_0}
=
(1+o(1))e^{-\lambda_0\deg(v)
(1+{\bar p}_v)^{-1}}
\le 
(1+o(1))e^{-\lambda\deg(v)}.
\end{equation}

\section{Cover Time}
\label{SectionCoverTime}

In  this section we consider the simple random walk 
$\EW_u$ 
on typical $\G$ starting at a vertex $u\in \V$.  Given $\G$ and $u$ we denote by
$C_u$ the expected time taken for the walk to visit every vertex of $\G$. We show that whp $t_1\le C_u\le t_0(1+o(1))$. In the proof we choose 
$T=\Theta(\ln n)$ satisfying conditions (\ref{CooperFrieze1}) and Lemma \ref{LemNotVisit}(ii)
and use the short-hand notation
$x=mpe^{-np}$ 
and 
$y=npe^{-\lambda}$, $y_1=npe^{-\lambda_1}$.

\subsection{Upper bound}

For each $u\in\V$ and all $t\ge T$ we have, see (42) of 
\cite{CooperFrieze2008} that
%CooperFrieze_Giant],
% see also formula (1) of 
%[CooperFrieze_RSA2007_1-16]
\begin{equation}\label{EqCu}
C_u\le t+1+\sum_v\sum_{s\ge t}
\PraG{\A_s(v)}.
\end{equation}
For  $t_0$ and $\lambda$  defined in  \eqref{EqDeft}, \eqref{EqDefLambda}, we have   
by Lemma~\ref{LemNotVisit}, see also (\ref{R_T}),
\begin{align*}
\sum_{v\in \V}\sum_{s\ge t_0}\PraG{\A_s(v)}
&
\le 
(1+o(1))
\sum_v\sum_{s\ge t_0}(\left(1+p_v\right)^{-s}
+
o(e^{-s/(2C_0T)}))
\\
&
=
(1+o(1))
\sum_v(\left(1+p_v\right)^{-t_0}\frac{1}{1(1+p_v)^{-1}}
+
o(n^{-1})
\\
&
=
(1+o(1))
\sum_v\frac{mn^2p^2}{\deg(v)}e^{-\lambda\deg(v)}
+
o(n^{-1})
\\
&
=
(1+o(1))
\sum_k \Dk \frac{mn^2p^2}{k}e^{-\lambda k}
+o(n^{-1}). 
\end{align*}
In the third line we used \eqref{XI-28-2}, \eqref{XI-28-1}, and \eqref{EqAt0upper}.
We show below that the sums 
$$
S_i
:=
\sum_{k\in K_i} \Dk \frac{1}{k}e^{-\lambda k}
=
o(1),
\qquad i=1,2,3.
$$
These bounds 
imply
$\sum_{v\in \V}\sum_{s\ge t_0}\PraG{\A_s(v)}=o(t_0)$. 
Now 
 \eqref{EqCu} yields
%the upper bound 
$C_u\le t_0+1+o(t_0)$.

\noindent
We first estimate $S_1,S_2$. For 
 $(c-1)\le (\ln n)^{-1/3}$ we have, by  Fact~\ref{FactLambda}(iii) and {\bf P8a}, 
$$
S_1\le 20 (\ln\ln n)^2e^{-\frac{1}{4}\ln\ln n} = o(1),
\qquad
S_2\le \Delta (\ln n)^4e^{-\frac{21}{4}\ln\ln n} = o(1).
$$
For $(c-1)\ge (\ln n)^{-1/3}$ we have,  by  
{\bf P8a}, {\bf P8b} and Fact~\ref{FactLambda}(i), 
$$
S_1+S_2\le 
\Delta (\ln n)^{4} \frac{1}{(\ln n)^{1/2}}e^{-\lambda (\ln n)^{1/2}} = o(1).
$$

\noindent
Now we estimate $S_3$. 
By property {\bf P8},
\begin{equation}\label{EqUpperK3}
S_3
\le
 \frac{3}{2}\sum_{k=1}^{\Delta}\EDk \frac{1}{k}e^{-\lambda k}
\le
\frac{3}{2}n^{1-c}
\sum_{i=1}^{\infty}
(mpe^{-np})^{i}
\sum_{k=i}^{\infty}
\frac{{k\brace i}}{k\cdot k!} (npe^{-\lambda})^k.
\end{equation}
To estimate the inner sum we use the following inequalities shown in Appendix B below.
\begin{fact}\label{FactStirling}
 For $y>0$  we have 
$
\sum_{k=i}^{\infty}
{k\brace i}
\frac{y^k}{k!\cdot k}
\le 
\frac{i+1}{i^2\cdot i!}\frac{(e^{y}-1)^{i}}{y}
\le 4\frac{(e^{y}-1)^i}{(i+1)!y}
$.
For $0<e^y-1<1/2$ we have
$
\sum_{k=i}^{\infty}
{k\brace i}
\frac{y^k}{k!\cdot k}
\le 
\frac{3}{2}\frac{1}{i!\cdot i} \frac{(e^{y}-1)^{i+1}}{y}
\le 3\frac{(e^y-1)^{i+1}}{(i+1)!y}
$.
\end{fact}
\noindent
For $e^y-1<1/2$ we obtain
from (\ref{EqUpperK3}) using Fact~\ref{FactStirling} 
that
\begin{equation}
\nonumber
%label{2019-05-01} 
S_3
\le 
4.5n^{1-c}
\sum_{i=1}^{\infty}
\frac{\left(e^{y}-1\right)^{i+1}}{y(i+1)!}x^{i}
= 
\frac{4.5}{xy}n^{1-c} 
\sum_{i=1}^{\infty}
\frac{((e^y-1)x)^{i+1}}{(i+1)!}
\le
\frac{4.5}{xy}n^{1-c}e^{(e^y-1)x}.
%\\
%&=\frac{5}{mpe^{-np}npe^{-\lambda}}=o(1).
\end{equation}
Furthermore, relations 
\begin{equation}
\label{2019-05-01+}
x(e^{np}-1)=\bigl(1+O(n^{-1}\ln n)\bigr)c\ln n
\qquad
{\text{and}}
\qquad
\
\frac{e^{npe^{-\lambda}}-1}{e^{np}-1}=\frac{c-1}{c},
\end{equation}
see  (\ref{c}) and (\ref{EqDefLambda}), 
imply $n^{1-c}e^{(e^{y}-1)x}=1+o(1)$.
Finally, we  establish the bound $S_3=o(1)$ by showing that
$xy\to+\infty$. For small $y>0$ satisfying $e^y-1<1/2$ we have
$2y>(e^y-1)$. Now (\ref{2019-05-01+})
implies
\begin{displaymath}
2xy>x(e^y-1)
=
x(e^{np}-1)(c-1)/c
=
(1+o(1))(c-1)\ln n
\to
+\infty.
\end{displaymath}
For $e^y-1\ge 1/2$ we obtain
from (\ref{EqUpperK3}) using the first inequality of Fact~\ref{FactStirling} 
that
\begin{equation}
\nonumber
%\label{2019-05-01} 
S_3
\le 
6n^{1-c}
\sum_{i=1}^{\infty}
\frac{\left(e^{y}-1\right)^{i}}{y(i+1)!}x^{i}
\le 
\frac{6}{xy(e^y-1)}n^{1-c} 
e^{(e^y-1)x}
=
\frac{6}{xy(e^y-1)}(1+o(1)).
%\\
%&=\frac{5}{mpe^{-np}npe^{-\lambda}}=o(1).
\end{equation}
Now the right side is $O(x^{-1})=o(1)$, since 
$y(e^{y}-1)\ge (\ln(3/2))/2$ is bounded away from zero.
This completes the  proof of the upper bound.

\subsection{Lower bound}

Here we define a large set of special vertexes $\Set$ and 
show that {\whp} some vertexes from $\Set$ are visited by 
$\EW_u$  only after time $t_1$, i.e., 
in the last phase before covering all the vertexes. 

\noindent
For $i_0$, $k_0$ and $\EDkiO$ defined in  {\bf P8}, let
\begin{displaymath}
I_1=\{k:i_0\le k\le k_0\text{ and } \EDkiO > k_0^{-2}e^{k\lambda_1}\},
\qquad
I_0=\{i_0,i_0+1,\dots, k_0\}\setminus I_1.
\end{displaymath}
Note that $k_0^{-2}e^{k\lambda_1}>i_0^2$ since  
$\lambda_1=(1+o(1))\lambda$ is 
bounded away from $0$ and $i_0\asymp k_0$. Hence
$I_1\subset I$,  where $I$ is from  {\bf P8}. 
Define the set of special vertexes 
\begin{equation*}
\Set
=
\left\{
v\in\V: |\W'(v)|=i_0,\, \deg(v)\in I_1,
\,
\min_{v':|\W'(v')|=i_0, v'\not=v}
{\dist}(v,v')\ge \frac{\ln n}{(\ln\ln n)^3} 
\right\}
%\setminus\{v\in\V:\exists_{v'\in\V} |\W'(v')|=i_0, %dist(v,v')< \ln n/(\ln\ln n)^3\},
\end{equation*}
and let $X$ be the number of vertexes in $\Set$ 
that are not visited in steps  $T,T+1,\ldots,t_1$. Note 
that $I_1\subset I$ implies
$
|\Set|=\sum_{k\in I_1}\Dstar
$, see {\bf P8}. 

\noindent
We show below that
$\EG X=\Omega(\ln^9n)$ and 
$\EG X^2-(\EG X)^2=o((\EG X)^2)$
uniformly over typical $\G$ and $u\in \V$.
These bounds yield 
$\PraG{2X>\EG X}\to 1$, by 
 Chebyshev's inequality. Hence whp 
 $X
=\Omega(\ln^9n)$. 
As the number of vertexes visited within the first $T$ steps is at most $T=O(\ln n)$ we will find in $\Set$ at least $X-T=\Omega(\ln^9 n)$ vertexes unvisited by the time $t_1$. Thus $C_u\ge t_1$.

Let us prove that $\EG X=\Omega(\ln^9 n)$.
It follows from (\ref{A++}) and {\bf P8c}  that
\begin{eqnarray}
\nonumber
\EG X
=
\sum_{v\in\Set}\PraG{\A_{t_1}(v)}
&
=
&
(1+o(1))
\sum_{k\in I_1}
\Dstar e^{-k\lambda_1}+o(1)
\\
\nonumber
&
\ge 
&
\Bigl(\frac{1}{2}+o(1)\Bigr)
\sum_{k\in I_1}
\bar{D}(k,i_0) e^{-k\lambda_1}+o(1).
\end{eqnarray}
We write the sum $\sum_{k\in I_1}
\bar{D}(k,i_0) e^{-k\lambda_1}$ in the form
\begin{displaymath}
\left(
\,
\sum_{k\ge i_0}-\sum_{k\ge k_0}-\sum_{k\in I_0}
\,
\right)
\bar{D}(k,i_0) e^{-k\lambda_1}
=:
S^*-S^*_0-S^*_1
\end{displaymath}
and show that $S^*=\Omega(\ln^9n)$ and $S^*_i=o(1)$, 
$i=0,1$. 
The bound $S^*_1\le |I_0|k_0^{-2}=o(1)$ is obvious.
Let us prove that $S^*_0=o(1)$. In the proof we use inequalities
\begin{equation}
\label{2019-05-02++1}
{k\brace i_0}\frac{1}{k!}
\le
\frac{(ke)^{i_0}i_0^k}{i_0^{2i_0}}
\frac{e^{k_0}}{k^{k_0}(k-k_0)!}
\le
\frac{(k_0e)^{i_0}}{i_0^{2i_0}}
\frac{e^{k_0}i_0^{k_0}}{k_0^{k_0}(k-k_0)!}i_0^{k-k_0}.
\end{equation}
The second inequality follows by $k_0\le k$. To get the first one we combine the inequalities
\begin{displaymath}
{k\brace i_0}
\le 
{\binom{k}{i_0}}
i_0^{k-i_0}
\le 
(ke)^{i_0}i_0^{k-2i_0}
\qquad
{\text{and}}
\qquad
k!\ge (k-k_0)!(k/e)^{k_0}
\end{displaymath}
that follow from
(\ref{2019-04-01}), $\binom{k}{s}\le (ke/s)^s$
and $k!/(k-s)!\ge (k/e)^s$ respectively
(the last inequality follows by induction on $s$).
From (\ref{2019-05-02++1}) we obtain
\begin{displaymath}
{\bar D}(k,i_0)e^{-k\lambda_1}
\le
\frac{e^{i_0}}{n^{c-1}}
\,
\left(
\frac{xk_0}{i_0^{2}}
\right)^{i_0}
\left(
\frac{ei_0y_1}{k_0}
\right)^{k_0}
\, 
\frac{(i_0y_1)^{k-k_0}}{(k-k_0)!}.
\end{displaymath}
Note that the first factor $e^{i_0}n^{1-c}\le e$.
Now summing over $k\ge k_0$ gives
 \begin{displaymath}
 S^*_0
 \le 
 e
 \cdot
 \left(
\frac{xk_0}{i_0^{2}}
\right)^{i_0}
\left(
\frac{ei_0y_1}{k_0}
\right)^{k_0}
 e^{i_0y_1}
 =
 e
 \cdot
 \left(
\frac{xe^{y_1}k_0}{i_0^{2}}
\right)^{i_0}
\left(
\frac{ei_0y_1}{k_0}
\right)^{k_0}.
 \end{displaymath}
%We can neglect the factor $e$. 
Furthermore, using $e^{y_1}=Aa+1$, $y_1\le Aa$, and the first relation of (\ref{2019-05-01+}) we upper bound $S^*_0/e$
by
 \begin{equation}\label{2019-05-02++3}
 \left(
 \frac{(1+O(n^{-1}\ln n))}{e^{np}-1}
 \frac{c}{c-1}\frac{k_0}{i_0}(Aa+1)
 \right)^{i_0}
 \cdot
 \left(
 \frac{ei_0}{k_0}Aa
 \right)^{k_0}.
 \end{equation}
 Now assume that $c-1\le (\ln\ln n)^2/\ln n$. In this 
 case our condition $(c-1)\ln n\to+\infty$ implies
 $Aa+1=O(1)$. Using $i_0\asymp k_0$ we
 upper bound 
 (\ref{2019-05-02++3}) by
 \begin{displaymath}
 \bigl(\Theta(1)\bigr)^{i_0+k_0}
 \left(
\frac{1}{c-1}
\right)^{i_0}
(Aa)^{k_0}
=
\bigl(\Theta(1)\bigr)^{k_0}A^{k_0}(c-1)^{k_0-i_0}=o(1).
\end{displaymath}
In the first  step we used $a/(c-1)=\Theta(1)$. In the last 
step we used $k_0-i_0\asymp k_0$ and 
$A^{k_0}=O(e^{O(\ln\ln n)})$. This shows $S^*_0=o(1)$.

\noindent
Next, assume that $c-1> (\ln\ln n)^2/\ln n$.
Using $e^s\le 1+2s$ for small 
$s=10(\ln\ln n)/((c-1)\ln n)$ we bound $A\le 1+2s$. Now, the inequality $(1+2s)(c-1)/c\le 1$, which holds for $c=O(1)$, yields $aA\le e^{np}-1$. Furthermore, a crude upper bound $A\le 3/e$ yields $aA\le (e^{np}-1)(c-1)c^{-1}(3/e)$. Invoking these upper bounds for $aA$
in the first and second factors  of (\ref{2019-05-02++3}) 
%respectively 
and using $(1+O(n^{-1}\ln n))^{i_0}=1+o(1)\le 2$
we upper bound (\ref{2019-05-02++3}) by
\begin{eqnarray}
\nonumber
2
 \left(
\frac{k_0}{i_0}
 \frac{e^{np}}{e^{np}-1}
 \right)^{i_0}
 \left(
 3\frac{i_0}{k_0}(e^{np}-1)
 \right)^{k_0}
 \left( 
 \frac{c-1}{c}
 \right)^{k_0-i_0}
 \le
 2\left(
 \frac{i_0}{k_0}e^{np}(e^{np}-1)
 \right)^{k_0-i_0}3^{k_0}
 =o(1).
 \end{eqnarray}
In the first step we used $(c-1)/c<1$ and 
$(e^{np})^{k_0-2i_0}\ge 1$.  This proves $S^*_0=o(1)$.
 
Let us prove that $S^*=\Omega(\ln^9n)$. By properties of Stirling's numbers we have
\begin{eqnarray}
S^*
=
n^{1-c}
x^{i_0}
\sum_{k\ge i_0}
{k\brace i_0}
\frac{y_1^k}{k!}
=
n^{1-c}
x^{i_0}
\frac{(e^{y_1}-1)^{i_0}}{i_0!}.
\end{eqnarray}
Furthermore, using Stirling's approximation to $i_0!$
and invoking the relations
\begin{displaymath}
n^{1-c}\asymp e^{-i_0},
\qquad
e^{y_1}-1=A(e^{np}-1)(c-1)/c,
\qquad
x(e^{np}-1)=mp(1-e^{-np})
\end{displaymath}
we get
\begin{displaymath}
S^*
\asymp 
\frac{1}{\sqrt{i_0}}
A^{i_0}
\left(
\frac{mp(1-e^{-np})}{i_0}
\frac{c-1}{c}
\right)^{i_0}
\asymp
\frac{1}{\sqrt{i_0}}
A^{i_0}
=\Omega(\ln^9 n).
\end{displaymath}
\noindent
Finally, we  show that $\EG X^2-(\EG X)^2=o((\EG X)^2)$. 
We have 
%$X=\sum_{v\in \Set}{\mathbb I}_{v}$ and
$X(X-1)=\sum_{\{u,v\}\subset \Set}{\mathbb I}_u{\mathbb I}_v$, where ${\mathbb I}_v$ denotes the indicator of 
event ${\cal A}_{t_1}(v)$.
% that the walk does not visit vertex $x$ in steps 
%$T,T+1,\dots, t_1$. 
By (\ref{2019-04-26}),
\begin{displaymath}
\EG {\mathbb I}_x{\mathbb I}_y
=
(1+o(1))
\PraG{\A_{t_1}(u)}\PraG{\A_{t_1}(v)}
+o(n^{-3}).
\end{displaymath}
Hence 
 $\EG X(X-1)=(1+o(1))(\EG X)^2$. 
 Finally,  for $\EG X\to +\infty$
  we obtain 
 \begin{displaymath}
 \EG X^2-(\EG  X)^2
 =
 \EG X(X-1)+\EG X-(\EG X)^2=o(\EG X)^2.
 \end{displaymath}

\section{Conclusions}

We determined the expected cover time of a random walk in random intersection graph $\Gnmp$ above its connectivity threshold. Our results were compared with corresponding results obtained for Erd\H{o}s--R\'{e}nyi random graph model. This comparison led us to conclusion that the presence of clustering and specific degree distribution in affiliation networks delay the covering of the network by a random walk (with relation to the random walk on $G(n,q)$ with corresponding edge density). 

We studied the random intersection graph model introduced by Karo\'nski at al. in \cite{karonski1999}. However various different random intersection graph models have been studied since, for example, in the context of security of wireless sensor networks \cite{BlackburnGerke2009,YoganMakowski2012}   or scale free networks \cite{Bloznelis2013} (see also \cite{BloznelisGJKR2015}, \cite{FriezeKaronski2016},  and \cite{Spirakisetal2013} for more models and applications). In the context of obtained results it would be intriguing to study the cover time in other models of random intersection graphs in order to understand more the relation between clustering, degree distribution, and the expected cover time of a random walk in real life networks.

\section*{Appendix}

\appendix

\section{Proof of Lemma~\ref{LemProperties}}

Before we proceed to the proof of {\bf P0-P7}
we collect several auxiliary results.
In the proof we use  the following version of Chernoff's inequality, see (2.6) in \cite{JansonLuczakRucinski2001}.
\begin{lem}\label{LemChernoff}
	Let $X$ be a random variable with the binomial distribution and expected value~$\mu$, then for any $0<\eps<1$ 
	$$
	\Pra{X\le \eps \mu}\le \exp\left(-\psi(\eps)\mu\right),
	\text{
		where
	}
	\psi(\eps)=\eps\ln \eps + 1 - \eps. 
	$$
\end{lem}

\begin{fact}\label{FactSmallCycles}
	There exists a
	constant  $\adistance>0$  depending on the sequences $\{p(m,n)\}$ and $\{c(n,m)\}$ such that {\whp} 
	\begin{itemize}
		\item[(i)] any two $\B$--cycles of length at most $\adistance \ln n/\ln\ln n$ are at least $\adistance \ln n/\ln\ln n$ links apart from each other. 
		\item[(ii)] any two small vertexes are at least $\adistance \ln n/\ln\ln n$ links apart from each other.
		\item[(iii)] any  small vertex is at least $\adistance \ln n/\ln\ln n$ links apart from any $\B$--cycle of length at most $\adistance \ln n/\ln\ln n$.	
	\end{itemize}
\end{fact}	 

\begin{fact}\label{Fact4cycles}
	With high probability there are at most $\ln^3n$ $\B$--cycles consisting of $4$ links.
\end{fact}

\begin{fact}\label{FactVw}	
	With high probability
	$|\V(w)|\le \frac{\ln n}{\ln\ln n}\max\{2,np\}$ for 
	every $w\in\W$.
	
\end{fact}

\begin{proof}[Proof of Fact \ref{FactSmallCycles}]
	Given $\adistance>0$ 
	let $j_0:=\lfloor \adistance \ln n/\ln\ln n\rfloor$.  
	
	\noindent
	Proof of (i). The expected number of pairs of $\B$-cycles 
	$C_k$, $C_r$ of length $2k$ and $2r$ which are connected by a (shortest) path of length $i\ge 1$ (links)
	containing
	$i_1$ internal  vertexes and  $i_2$ 
	internal attributes  
	(those outside $C_k$ and $C_r$)  
	is at most
	\begin{equation}\label{2018-10-02}
	{\binom{n}{k}}{\binom{m}{k}}p^{2k}
	{\binom{n}{r}}{\binom{m}{r}}p^{2r}
	{\binom{n}{i_1}}{\binom{m}{i_2}}p^{i}
	4(k!r!)^2i_1!i_2!
	\le 
	4n^{k+r+i_1}m^{k+r+i_2}p^{2k+2r+i}.
	\end{equation}
	Note that $i=i_1+i_2+1$ and for 
	$i=1$ we have  $i_1=i_2=0$. 
	Assuming that $n\le m$ and $mp\ge 1$ 
	we see that  the sum of (\ref{2018-10-02}) 
	over
	$2\le k,r\le j_0$ and $0\le i_1+i_2\le j_0$ 
	is at most
	\begin{equation}\label{2019-05-09+1}
	\sum_{k=2}^{j_0}
	\sum_{r=2}^{j_0}
	\sum_{i_1=0}^{j_0}
	\sum_{i_2=0}^{j_0}
	4p(mp)^{2k+2r+i-1}
	\le 
	4j_0^4p(mp)^{5j_0}
	=o(1).
	\end{equation}
	To achieve the last bound we choose 
	$\adistance>0$ sufficiently small. 
	We obtain that the expected number of pairs of $\B$-cycles of length at most $2j_0$  that are in the distance $d\in [1,j_0]$ is $o(1)$. 
	Hence {\whp} we do not observe such a pair.
	
	\noindent
	Next we count pairs of $\B$--cycles $C_k$, $C_r$ that share at least one vertex or attribute. For any $\B$--cycle $C$ we denote by $\V_C$ the set of its vertexes and by $\W_C$ the set of its attributes. Let $u\in \bigl(\V_{C_r}\cup \W_{C_r}\bigr)\setminus \bigl(\V_{C_k}\cup \W_{C_k}\bigr)$. We can walk from $u$ along $C_r$ in two directions until we reach  the set 
	$\V_{C_k}\cup \W_{C_k}$. In this way we obtain a path belonging to $C_r$ and with endpoints in 
	$\V_{C_k}\cup \W_{C_k}$. Internal vertexes/attributes of the path do not belong to $C_k$.
	By $i$ we denote the length of the path (number of links). 
	$i_1$ and $i_2$ are the numbers of internal  vertexes and attributes of the path  ($i=i_1+i_2+1$, $|i_1-i_2|\le 1$).
	The union of $C_k$ and the path defines an 
	eared $\B$-cycle. The expected number of such eared cycles is at most
	\begin{equation}\label{2018-10-02+1}
	{\binom{n}{k}}{\binom{m}{k}}p^{2k}{\binom{n}{i_1}}{\binom{m}{i_2}}p^{i} 4k(k!)^2i_1!i_2!
	\le 
	4kn^{k+i_1}m^{k+i_2}p^{2k+i}
	\end{equation} 
	Assuming that $n\le m$ and $mp\ge 1$ 
	we see that  the sum of (\ref{2018-10-02+1}) 
	over
	$2\le k\le j_0$ and $2\le i\le 2j_0$ 
	is at most 
	\begin{displaymath}
	\sum_{k=2}^{j_0}
	\sum_{i_1=1}^{j_0}
	\sum_{i_2=1}^{j_0}
	4kp(mp)^{2k+i-1}
	\le 
	4j_0^4p(mp)^{4j_0}
	=o(1).
	\end{displaymath} 
	The last bound follows by our choice of 
	$\adistance>0$ in (\ref{2019-05-09+1}). 
	We obtain that the expected number of eared $\B$-cycles (where the cycle and the ear have at most $j_0$ vertexes and $j_0$ attributes  each) is $o(1)$. 
	Hence {\whp} we do not observe a pair of $C_k$, $C_r$ with $2\le k,r\le j_0$ that share at least one vertex or attribute.	
	\medskip

	\noindent  Proof of (ii). 
	Assume that there exist two vertexes  $v,v'\in\SMALLv$, 
	which are $2t$ links apart  
	blue
	in $\B$ for some $2t\le j_0$. Then blue $\B$ contains a 
	path $vw_1v_1w_2\ldots v_{t-1}w_tv'$ of length $2t$ and 
	the set 
	$(\W'(v)\cup\W'(v'))\setminus \{w_1,\ldots,w_t\}$ is of 
	cardinality at most $0.2\ln n$.
	The number of possible paths of the form  
	$vw_1v_1w_2\ldots v_{t-1}w_tv'$ is at most 
	$n^{t+1}m^{t}$ and the probability that a path of 
	length $2t$ is present in $\Bnmp$ is $p^{2t}$. 
	Furthermore, given the event that the path  
	$vw_1v_1w_2\ldots v_{t-1}w_tv'$ is present, the 
	cardinality of the set $(\W'(v)\cup\W'(v'))\setminus \{w_1,\ldots,w_t\}$ has the binomial distribution 
	$\Bin{m-t}{1-(1-p)^2}$. By Lemma~
	\ref{LemChernoff}, 
	this cardinality is less than 
	$0.2\ln n$ with probability at most $e^{-1.2\ln n}$. Indeed,
	the binomial distribution has the expected value $\mu=2mp+o(1)>2c\ln n>2\ln n$.
	Finally, by the union bound
	the probability that there exist two small vertexes within the distance $j_0$ (links)  
	is at most
	\begin{displaymath}
	\sum_{t=1}^{j_0/2}n^{t+1}m^{t}p^{2t}e^{-1.2\ln n}
	\le 
	j_0 n (nmp^2)^{j_0/2} e^{-1.2\ln n}  
	= o(1).
	\end{displaymath}
	\medskip
	
	\noindent 
	Proof of (iii)  is a combination of
	those of  (i) and (ii). 
\end{proof} 

\begin{proof}[Proof of Fact~\ref{Fact4cycles}]
	The expected number of $\B$--cycles consisting of $4$ links is at most
	$n^2m^2p^4 =O(\ln^2 n)$.
	The fact follows by  Markov's inequality.
\end{proof}

\begin{proof}[Proof of Fact \ref{FactVw}]
For shortness let
	$s
	=
	\max\{2,np\}(\ln n)/\ln\ln n$. By the union bound,
	%  We have
	\begin{align*}
	\Pra{\exists_{w\in \W} |\V(w)|\ge s}
	&\le 
	m
	\Pra{|\V(w)|\ge s}
	\le m\binom{n}{s}p^{s}
	\le m \left(\frac{enp}{s}\right)^{s}
	\\
	&\le 
	\exp
	\left(
	\ln m - s\ln s + O(s)
	\right)=o(1).
	\end{align*}		
\end{proof}

\noindent
Proof of {\bf P0-P7}.

\medskip

\noindent{\bf P0} Any triple of 
vertexes  linked to some $w$ in $\B(n,m,p)$ induce the triangle 
in  $G(n,m,p)$. The degrees of attributes $w\in W$ in  $\B(n,m,p)$
are independent Binomial random variables with mean 
$pn=\Theta(1)$.  Therefore, the   maximal degree is greater than 2 whp.
Hence $\G(n,m,p)$ contains a triangle whp.
For the connectivity property we refer to 
\cite{Katarzyna2017}. 
\medskip

\noindent {\bf P1} 
Let $X$ be  binomial $\Bin{n}{p}$ random variable and let
$$
Y_1=\sum_{w\in\W}\binom{|\V(w)|}{2},
\qquad
Y_2=\sum_{w,w'}\binom{|\V(w)\cap \V(w')|}{2}.
$$
We have, by the inclusion-exclusion argument, 
\begin{equation}
\label{EqEdges}
Y_1 - Y_2 \le |{\cal E}(\Gnmp)|\le Y_1.
\end{equation} 
Note that $Y_2$ is 
the number of $\B$--cycles with $4$ links. Hence
$Y_2\le \ln^3n$ {\whp}, by Fact~\ref{Fact4cycles}.
Furthermore, as $|\V(w)|$, $w\in\W$, are independent 
copies of $X$,
we obtain for $np=\Theta(1)$ that
\begin{align*}
\E Y_1 &=
m\frac{\E X(X-1)}{2}
=  
m\frac{n(n-1)p^2}{2}=
% \frac{n^2mp^2}{2}-\frac{nmp^2}{2}=
m\cdot\Theta(1)=\Theta(n\ln n),
\\
\Var Y_1 &= m \frac{\E X^2(X-1)^2 - (n(n-1)p^2)^2}{2}
= 
m\cdot \Theta(1)=\Theta(n\ln n).
\end{align*}
Now, Chebyshev's inequality implies
$
\Pr\{|Y_1-\E Y_1|> \sqrt{\E Y}\ln^{1/4}n\}=o(1)
$.
This bound combined with $Y_2=O(\ln^3n)$ and 
\eqref{EqEdges}
shows {\bf P1}.

\medskip

\noindent {\bf P3} Proof of $\deg(v)\le \Delta$.
Given $A\subset \W$ of size 
$|A|\le 4mp$, let $Y$ be the number of vertexes $v\in \V$ linked to at least one attribute from $A$. The random variable $Y$ has binomial distribution 
$\Bin{n}{1-(1-p)^{|A|}}$.
Let $Y'$ be a random variable with the distribution
$\Bin{n}{4d_1/n}$. 
Inequalities 
$(1-p)^{|A|}
\ge 
(1-p)^{ 4mp}\ge 1-4d_1/n$
imply that $Y$ is stochastically dominated by $Y'$, 
i.e., $\Pr\{Y>s\}\le \Pr\{Y'>s\}$, $\forall s>0$.
We have
%
%
% will show the statement only for $\deg(v)$. $|\W(v)|$ % and $|\W'(v)|$ have the binomial distribution 
% $\Bin{m}{p}$ and $\Bin{m}{\dO/m}$, resp.. Therefore 
% the proofs of the remaining statements are similar
% but easier. In the latter case, given 
% $|\W(v)|\le 4mp$, 
% $N(v)$ is stochastically dominated by a binomial
% random variable with the binomial distribution 
% $\Bin{n}{4\djeden/n}$ 
%(as $1-(1-p)^{|\W(v)|}\le 1-(1-p)^{4mp}\le 4mp^2$).
% Having in mind that $\djeden\ge \dO\ge  \ln n$ and $mp\ge \dO\ge \ln n$
\begin{align*}
\Pr\{\deg(v)> 12d_1\}
&
\le 
\Pr\bigl\{\deg(v)> 12d_1\bigl|\,|\W(v)|\le 4mp\bigr\}
+
\Pr\{|\W(v)|>4mp\}
\\
&
\le
\Pr\{Y'>12d_1\}+\Pr\{|\W(v)|>4mp\}
\\
&\le 
\binom{n}{\lceil 12\djeden\rceil}\left(\frac{4\djeden}{n}\right)^{\lceil 12\djeden\rceil}
+
\binom{m}{\lceil 4mp\rceil}p^{\lceil 4mp\rceil}
\\
&\le 
\left(\frac{e}{3}\right)^{\lceil 12\djeden\rceil}
+\left(\frac{e}{4}\right)^{\lceil 4mp\rceil}
=o(n^{-1}).
\end{align*}
Now, by the union bound $\Pra{\exists_{v\in \V} \deg(v)> 12\djeden}
\le
n\Pra{ \deg(v)>12\djeden}=o(1)$. We conclude that $\max_{v\in \V}\deg(v)> \Delta$ {\whp}.

\noindent
Proof of $|\W'(v)|\le \Delta$.
The random variable $|\W'(v)|$ has binomial distribution $\Bin{m}{p_*}$ with
$p_*=p(1-(1-p)^{n-1})=d_0/m$. We have
\begin{align*}
\Pr\{|\W'(v)|>4d_0\}
\le 
{\binom{m}{\lceil 4d_0\rceil}}
p_*^{\lceil 4d_0\rceil}
\le  
\left(\frac{emp_*}{\lceil 4d_0\rceil}
\right)^{\lceil 4d_0\rceil}
=o(n^{-1}).
\end{align*}
Hence,
$\Pra{\exists_{v\in \V} |\W'(v)|> 4d_0}
\le
n\Pra{|\W'(v)|> 4d_0}=o(1)$.

\medskip
\noindent{\bf P4} Any pair of adjacent  vertexes $v,v'\in \V$ share at most two common attributes (otherwise there were two intersecting $\B$-cycles of length $4$, the event ruled out by Fact~\ref{FactSmallCycles}(i)). Assume that $v,v'$ share two  attributes $w, w'$. In this case all common neighbors of  $v,v'$ belong to 
$\V(w)\cup\V(w')$
(otherwise there were two intersecting $\B$-cycles of length at most $6$). Now  
{\bf P4} follows from Fact~\ref{FactVw}.
Next, assume that $v,v'$ share only one attribute $w$. In this case there might be at most one common neighbor of $v,v'$ outside $\V(w)$ (otherwise there were two intersecting $\B$-cycles of length at most $6$). Hence the number of common neighbors is at most $|\V(w)|-2+1$ and we obtain {\bf P4} from 
Fact~\ref{FactVw}.
\medskip

\noindent {\bf P5} If  $\deg(v)\le |\W'(v)|-2$ then we either find $u\in \V\setminus \{v\}$ linked to at least three different elements of $\W'(v)$ or we find 
$u_1,u_2\in\V\setminus \{v\}$ such that $u_i$ is linked to at least two elements of $\W'(v)$ for each $i=1,2$. In both cases
there is a pair of short $\B$-cycles 
containing $v$, the event ruled out by Fact~\ref{FactSmallCycles}(i). Hence $d(v)\ge |\W'(v)|-1$ {\whp}. For a large vertex $v$, the latter inequality implies $\deg(v)\ge (\ln n)/11$.
\medskip

\noindent{\bf P6} Assume that we find $v\in \V$ and  $u_1^*, u_2^*, u_3^*\in N_{i-1}(v)$ and
$u\in N_i(v)$ such that $u$ is adjacent to  each $u_1^*, u_2^*, u_3^*$. Vertex 
$u$ is $2i$ links apart from $v$ in $\B$.
Furthermore, each $u_j^*$ is  $2(i-1)$ links apart from 
$v$ and $2$ links apart from $u$, for $j=1,2,3$.  Therefore we find  three distinct  shortest paths  connecting $u$ and $v$ in $\B$ (via $u_1^*, u_2^*$ and $u_3^*$). These paths create at least two short $\B$-cycles close to $u$. But, by Fact~\ref{FactSmallCycles}(i),  there are no such cycles {\whp}. 
\medskip

\noindent{\bf P7} follows by Fact~\ref{FactSmallCycles}.

%\noindent{\bf P8c} 

\section{Proof of Facts \ref{FactLambda},  \ref{FactStirling}}

\begin{proof}[Proof of Fact~\ref{FactLambda}]
	Proof of (i). For $y>x>0$ we have 
	$\ln (1+y)-\ln (1+x)> (y-x)\ln'(1+y)$, since $x\to\ln'(1+x)$ is decreasing.  We apply 
	$\ln(1+y)-(y-x)(1+y)^{-1}>\ln(1+x)$ to
	$x=a$ and $y=e^{np}-1$ and obtain
	\begin{displaymath}
	\ln(1+a)<np-(1-e^{-np})/c.
	\end{displaymath}	
	For $np=\Theta(1)$ and $1<c=O(1)$ we find  and absolute 
	constant $\delta>0$ such that $(1-e^{-np})/c>\delta$. 
	Hence
	$\lambda>\ln(np/(np-\delta))$ is bounded away from 
	zero.

	\noindent
	Proof of (ii).
	$(c-1)\ln n\to\infty$ implies $a>2\ln^{-1}n$ for large $n$. Furthermore,  $\ln(x+1)\ge x/2$, 
	for $0<x<2$, implies
	$\ln (a+1)>\ln^{-1}n$. Therefore
	$$
	\lambda
	\le 
	\ln\bigl(np/\ln n\bigr)  
	\le  
	\ln np 
	+ 
	\ln\ln n
	\le 2\ln\ln n.
	$$
	
	\noindent 
	Proof of (iii). Using $\ln (1+a)\le a$ 
	we bound from below
	\begin{equation}\label{2019-04-26++1}
	\lambda
	\ge
	\ln (np/a)
	=
	\ln \frac{cnp}{e^{np}-1}
	+
	\ln((c-1)^{-1})
	= 
	O(1)
	+
	\ln((c-1)^{-1}).
	\end{equation} 
	For $c-1\le \ln^{-1/3}n$ and large $n$ the right side is at least
	$4^{-1}\ln\ln n$.	
	
	\noindent 
	Proof of (iv). Our assumptions $np=\Theta(1)$ and $1<c=O(1)$ imply that  the lower bound of 
	$\ln \frac{cnp}{e^{np}-1}$, denoted by $b$, is finite. 
	Let us first show that $1+|\ln(c-1)|=O(\lambda)$.
	% and Let ${\bar b}=-\min\{b, -1\}$. 
	For 
	$1<c<1.5$ and $|\ln (c-1)|>2\max\{-b, 0\}$ 
	this relation follows from (\ref{2019-04-26++1}). Otherwise it follows from (i).
	Now we show that $\lambda=O(1+|\ln(c-1)|)$.
	For $c\le \min\{1.5,(e^{np}-1)^{-1}\}$ we have $a\le1$ and  
	inequality $\ln(1+a)\ge a/2$ implies
	\begin{equation}
	\nonumber
	\lambda
	\le
	\ln (2np/a)
	=
	\ln \frac{2cnp}{e^{np}-1}
	+
	\ln((c-1)^{-1})
	= 
	O(1)
	+
	\ln((c-1)^{-1}).
	\end{equation} 
	For $c> \min\{1.5,(e^{np}-1)^{-1}\}$ we have 
	$a=\Theta(1)$. This implies $\lambda=O(1)$.

	\noindent
	Proof of (v).
	In view of
	$\lambda>\lambda_1$ it suffices to show that 
	$\lambda
	\le \lambda_1+o(\lambda)$.
	We firstly assume that 
	$(c-1)\ln n\le (\ln\ln n)^2$. By the inequalities
	$x(1-x)\le \ln(1+x)\le x$, we have
	\begin{displaymath}
	\lambda-\lambda_1
	=
	\ln \frac{\ln \left(aA+1\right)}{\ln \left(a+1\right)}\le \ln \frac{aA}{a(1-a)}
	=\ln A -\ln (1-a)= \ln A+o(1).
	\end{displaymath}
	Furthermore, 
	%$(c-1)\ln n\to+\infty$ implies
	
	\begin{displaymath}
	\ln A
	%=
	%\frac{10\ln\ln n}{(c-1)\ln n}
	= 
	\ln (c-1)^{-1}\frac{10}{(c-1)\ln n}
	+\frac{10\ln((c-1)\ln n)}{(c-1)\ln n}=o(\ln (c-1)^{-1})+o(1).
	\end{displaymath}
	In the last step we used $(c-1)\ln n\to+\infty$.
	We obtain $\lambda-\lambda_1=o(|\ln(c-1)|+1)$.
	Now (v) follows from (iv). 
	
	\noindent
	Next we assume that
	$(c-1)\ln n> (\ln\ln n)^2$. In this case $\ln A\le 10/\ln\ln n$. By the inequality
	$e^x-1\le 2x$, for  $0\le x\le 1/2$, we have
	$$
	A-1
	\le
	2\ln A
	\le
	20/\ln\ln n.
	$$
	Set $f(x)=\ln\ln (ax+1)$. Then $f'(x)=a(\ln (ax+1))^{-1}(ax+1)^{-1}$ is a decreasing function, 
	for $x\ge 0$. From this fact and the inequalities 
	$A>1$ and $a/(1+a)\le \ln(1+a)$ we obtain
	\begin{align*}
	\lambda-\lambda_1
	&=
	\ln\ln (aA+1)-\ln\ln (a+1)
	= 
	f(A)-f(1)\\
	&
	\le f'(1)(A-1)
	=\frac{1}{\ln (a+1)}\frac{a}{1+a}(A-1)\le A-1 \le \frac{20}{\ln\ln n} = o(\lambda).
	\end{align*}
	In the last step we used  (i).

	\noindent
	Proof of (vi). We  write for short $x=np$, $b=(c-1)/c$, $z=e^{x}$.
	To prove the first inequality
	we show that $x(c\kappa-1)/(c\kappa)>\ln (b(e^x-1)+1)$.
	To this aim we establish the same inequality, but for  the respective derivatives $\partial/\partial x$. Indeed,  the derivative of the left side
	$1-e^{-x}(1-b)$ is greater than  that  of the right side $(b(e^x-1)+1)^{-1}be^{x}$ because  we have $(1-(1-b)z^{-1})(b(z-1)+1)>bz$, for $0<b<1$ and $z>1$.
	
	\noindent
	To prove the second inequality we show that $x\ln^{-1}(b(e^x-1)+1)<b^{-1}$. We have
	$xb<\ln(b(e^x-1)+1)$ because the respective inequality holds for the derivatives $\partial/\partial x$.

\end{proof}

\begin{proof}[Proof of Fact~\ref{FactStirling}]
	We start with auxiliary
	%Consider the 
	function
	$f_i(t)=
	t+\sum_{j=1}^i(1-e^t)^jj^{-1}$, $i=1,2,\dots$. 
	Its derivative $f'_i(t)=(1-e^t)^i$ satisfies $f_i'(t)(-1)^i=(e^t-1)^i>0$,  for $t>0$. Now  $f_i(0)=0$ implies $f_i(t)(-1)^i>0$ for $t>0$ and $\forall i$. Using this fact we obtain the upper bound
	\begin{equation}\label{2019-05-06}
	f_i(t)(-1)^i
	=
	(e^t-1)^ii^{-1}-f_{i-1}(t)(-1)^{i-1}
	\le 
	(e^t-1)^ii^{-1},
	\qquad
	t>0,
	\quad 
	\forall i.
	\end{equation}
	Furthermore, the identity $\ln(1+x)=-\sum_{j\ge 1}(-x)^{j}j^{-1}$ with $x=e^t-1$ implies
	\begin{equation}\label{2019-05-06+2}
	f_i(t)
	=
	\ln(1+x)+\sum_{1\le j\le i}(1-e^t)^jj^{-1}
	=
	-\sum_{j\ge i+1}(1-e^t)^jj^{-1}.
	\end{equation}
	\noindent
	Finally, the well known identity
	$
	\sum_{k=i}^{\infty}{k\brace i}\frac{t^k}{k!}=
	\frac{(e^t-1)^i}{i!}
	$ implies
	\begin{equation}\label{2019-05-06+1}
	f_i(y)\frac{(-1)^i}{i!}
	=
	\int_0^y
	\frac{(e^t-1)^i}{i!}
	dt
	=
	\sum_{k\ge i}{k\brace i}\frac{y^{k+1}}{(k+1)!}.
	\end{equation}
	The first inequality of Fact~\ref{FactStirling} 
	is an immediate consequence of 
	(\ref{2019-05-06}), (\ref{2019-05-06+1}).
	The second inequality of Fact~\ref{FactStirling} 
	follows from (\ref{2019-05-06+1}) combined with  
	(\ref{2019-05-06+2}). In this case we have
	\begin{displaymath}
	f_i(y)\frac{(-1)^i}{i!}
	=
	\frac{(-1)^{i+1}}{i!}
	\sum_{j\ge i+1}
	\frac{(1-e^{y})^j}{j}
	=
	\frac{(e^y-1)^{i+1}}{(i+1)!}\Bigl(1+(-1)^{i+1}R\Bigr),
	\end{displaymath}
	where $R=\sum_{k\ge 1}(1-e^y)^k(i+1)/(i+1+k)$. Note that $|e^y-1|<1/2$ implies $|R|\le 1/2$.
\end{proof}

\bibliographystyle{plain}
\normalfont
%\bibliography{Cover}

\end{document}